\documentclass[article,11pt,leqno]{amsart}
\usepackage{ucs}
\usepackage{upgreek}
\usepackage{amsthm,amsfonts,amssymb,epsfig,graphics,cancel,esint,upgreek,
amsmath,eufrak,latexsym,enumerate,mathrsfs}
\usepackage[T1]{fontenc}
\usepackage{lmodern}
\usepackage{oldgerm}
\usepackage[latin1]{inputenc}\relax
\usepackage[active]{srcltx}
\usepackage[final]{pdfpages}
\usepackage{color}
\usepackage{setspace}
\definecolor{darkred}{rgb}{0.5,0,0}
\definecolor{darkgreen}{rgb}{0,0.5,0}
\definecolor{darkblue}{rgb}{0,0,0.5}
\usepackage{graphicx,subfigure}
\usepackage[colorlinks = true, linkcolor=darkblue,
filecolor=darkgreen,
urlcolor=darkred,
citecolor=darkgreen,pdftex]{hyperref}
\usepackage{url}
\usepackage[small]{caption}
\usepackage{wrapfig}
\usepackage{enumitem}

\newcommand{\CC}{{\mathbb C}}

\newcommand{\LL}{{\mathbb L}}
\newcommand{\MM}{{\mathbb M}}
\newcommand{\NN}{{\mathbb N}}
\newcommand{\PP}{{\mathbb P}}
\newcommand{\QQ}{{\mathbb Q}}
\newcommand{\RR}{{\mathbb R}}

\newcommand{\TT}{{\mathbb T}}

\newcommand{\ZZ}{{\mathbb Z}}

\def\uu
\def\pp{\text {\rm \bf \c{p}}}

\def\cA{\mathcal{A}}

\def\cC{\mathcal{C}}
\def\cD{\mathcal{D}}
\def\cE{\mathcal{E}}
\def\cF{\mathcal{F}}
\def\cG{\mathcal{G}}
\def\cH{\mathcal{H}}

\def\cJ{\mathcal{J}}
\def\cK{\mathcal{K}}
\def\cL{\mathcal{L}}
\def\cM{\mathcal{M}}

\def\cO{\mathcal{O}}

\def\cR{\mathcal{R}}
\def\cS{\mathcal{S}}
\def\cT{\mathcal{T}}
\def\cU{\mathcal{U}}
\def\cV{\mathcal{V}}

\def\eps {\varepsilon}

\def\part{\partial}

\newcommand\id{\mathrm{Id} }

\newtheorem{theo}{Theorem}

\newtheorem{cor}[theo]{Corollary}
\newtheorem{lem}[theo]{Lemma}

\numberwithin{equation}{section}

\newcounter{numero}
\newcounter{numerob}
\newcounter{numerobb}
\usepackage{ulem}
\usepackage{cancel}
\usepackage{bbm}
\usepackage{scalerel,stackengine}
\usepackage{yhmath}

\newcommand{\myunderbar}[1]{\underline{#1\mkern-4mu}\mkern4mu }

\newcommand{\myoverbar}[1]{\mkern 1.5mu\overline{\mkern-3.3mu#1\mkern-.5mu}\mkern 1.5mu}
\newcommand{\myoverbarM}[1]{\mkern 1.5mu\overline{\mkern-5.5mu#1\mkern-.8mu}\mkern 1.5mu} 

\newcommand\rwidehat[1]{
  \savestack{\tmpbox}{\stretchto{
      \scaleto{
        \scalerel*[\widthof{\ensuremath{#1}}]{\kern.1pt\mathchar"0362\kern.1pt}
                 {\rule{0ex}{\textheight}}
                }{\textheight} 
    }{2.3ex}}
  \ensurestackMath{\stackon[-6.9pt]{#1}{\tmpbox}}
}\parskip 1ex

\usepackage{geometry}
\usepackage{amsmath}
\usepackage{amsfonts}
\usepackage{amsthm,amssymb}
\usepackage[latin1]{inputenc}  
\usepackage[T1]{fontenc}       
\usepackage{amssymb}

\title[Euler--Maxwell Two-Fluid]{
  The incompressible limit of the \\
  Euler--Maxwell two-fluid system
}
\author[N. Besse]{Nicolas Besse}
\address[Nicolas Besse]{Observatoire de la C\^ote d'Azur,
               Bd de l'Observatoire CS 34229,
               06304 Nice Cedex 4, France}
\author[C. Cheverry]{Christophe Cheverry}
\address[Christophe Cheverry]{Institut Math\'ematique de Rennes,
Campus de Beaulieu, 263 avenue du G\'en\'eral Leclerc CS 74205
35042 Rennes Cedex\\FRANCE}

\begin{document}

$ \, $

\maketitle

\noindent \textbf{\small Abstract.} {\small In this text, the filtering unitary group method developed, among others, by S. Schochet \cite{S07}, is adapted to prove the existence and well-posedness of modulation equations describing the incompressible limit of the Euler--Maxwell Two-Fluid (EMTF) system. The reduced model  captures up to the ion and electron skin depths the long-time  behavior of solutions near a constant neutral background with non-zero densities. In  the prepared case, the solutions of our  asymptotic equations are in one-to-one correspondence with those of incompressible eXtended  MagnetoHyDrodynamics (XMHD), hence providing a new basis to the XMHD framework  which is currently being studied by physicists through Hamiltonian methods, see P.J. Morrison \cite{KLMLW} {\it et al}. By this way, we can give a simplified access to many  plasma phenomena such as (a form of two-fluid) turbulence, Hall and   inertial effects, as well as collisionless magnetic reconnection.}

{\small \parskip=1pt
\setcounter{tocdepth}{2}
\tableofcontents
}


\section{Introduction}
\label{s:intro}
More than 99\% of the visible  universe is believed to be in the plasma state, consisting of a significant portion of charged particles (ions and electrons). The dynamics (in time $ \tau \in \RR $) of such conducting fluids is basically described by the Euler--Maxwell Two-Fluid system (EMTF) in space dimension three.

There is a general consensus that the two-fluid description of plasmas can be based on the following system of equations
\begin{equation}\label{TFEM}
  \left\{
  \begin{array}{l}
    \part_{\tau} {\rm n}_e + \nabla \cdot ({\rm n}_e \, {\rm v}_e) = 0 \,, \smallskip \\
    \part_{\tau} {\rm n}_i \, + \nabla \cdot ({\rm n}_i \,\, {\rm v}_i) = 0 \,, \smallskip \\
    \part_{\tau} ({\rm n}_e {\rm v}_e) +   \nabla \cdot ({\rm n}_e {\rm v}_e\otimes {\rm v}_e ) + \nabla {\rm p}_e/m_e
    \, = -\, {\rm n}_e \,  ( {\rm E} + {\rm v}_e \times {\rm B})/m_e \,, \smallskip \\
   \part_{\tau}  ({\rm n}_i {\rm v}_i) \, +  \nabla \cdot ({\rm n}_i \,{\rm v}_i \otimes {\rm v}_i)  \,+ \, \nabla {\rm p}_i/m_i
    \, = \,Z\, {\rm n}_i \, ({\rm E} + {\rm v}_i \times {\rm B})/m_i \,,\smallskip \\
    \displaystyle  \part_{\tau} {\rm E} \, - \, \nabla \times {\rm B} =   {\rm n}_e \, {\rm v}_e  - Z \,  {\rm n}_i \, {\rm v}_i \,, \smallskip \\
    \part_{\tau} {\rm B} \, + \, \nabla \times {\rm E} = 0 \,,
  \end{array}
  \right.
\end{equation}
completed with Gauss's laws
\begin{equation}\label{Gausslaw}
  \nabla \cdot {\rm B} = 0 \,, \qquad  \nabla \cdot {\rm E} \, + \, {\rm n}_e   - Z \,  {\rm n}_i = 0 \,,
\end{equation}
and with the barotropic closure
\begin{equation}\label{barotrop}
  {\rm p}_s = {p}_s( {\rm n}_s)\,, \quad  s \in \{e,i\}\,.
\end{equation}
The electrons and ions have respectively densities ${\rm n}_e$ and $ {\rm n}_i $, velocities $ {\rm v}_e $ and $ {\rm v}_i $, and dimensionless masses $ m_e $ and $m_i$ in $ \RR_+^* $. The parameter  $ Z $ is the charge number. From now on\footnote{By changing $ {\rm n}_i $, $ m_i $ and $ p_i $ into $ Z {\rm n}_i $, $ m_i /Z $ and $ p_i(\cdot/Z)$, we recover \eqref{TFEM}-\eqref{Gausslaw}-\eqref{barotrop} with $ Z = 1 $.}, we work with $ Z=1 $. The scalar functions $ {p}_s : \RR_+ \rightarrow \RR $, with $s \in \{e,i\}$,  represent  barotropic pressure laws satisfying $ {p}_s' > 0 $ on $ \RR_+^* $. The self-consistent electromagnetic field consists of the electric field $ {\rm E} $ and the magnetic field $ {\rm B} $. The complete unknonw is
\[
  {\rm U} = {}^t ({\rm n}_e,{\rm n}_i,{}^t{\rm v}_e,{}^t{\rm v}_i,{}^t{\rm E},{}^t{\rm B}) \in \RR_+ \times \RR_+ \times \RR^3 \times \RR^3 \times \RR^3 \times \RR^3\,.
\]
EMTF is a system of conservation laws, which is written above in standard dimensionless variables\footnote{The standard nondimensionalization consists in taking all dimensionless ratios of dimensional constants to one when performing the dimensional analysis. This amounts to setting the physical constants (such as the light velocity, the permitivity, the permeability and the electron charge) to one. But here the dimensionless positive masses $m_e$ and $m_i$ are kept in order to bear in mind the smallness of $ m_i $ and of the ratio $ \updelta := m_e/m_i $.}. Note that we do not add dissipation or relaxation terms that could improve the stability \cite{DLZ,ISYG,XXK} but that would also destroy or lessen the purely  hyperbolic properties that we aimed at  exhibiting. The flat neutral equilibrium
\begin{equation}
  \label{neutral}
  \myunderbar{U} = {}^t (\underline{n},\underline{n},0,0,0,0) \,, \qquad \underline{n} \in \RR_+ \,,
\end{equation}
is a constant global solution to \eqref{TFEM}-\eqref{Gausslaw}-\eqref{barotrop}. From now on, we fix $ s\in \RR $ with $ s>5/2  $, and we denote by $ H^s $ the $ L^2 $-based Sobolev space on $ \RR^3 $ or $ \TT^3 $. Given $ \eps \in ]0,\eps_0 ] $ with $ \eps_0 \in \RR_+^* $, at time $ {\tau}= 0 $, we modify $ \myunderbar{U} $ according to
\[
  {\rm U}(0,\cdot) = {\rm U}^0_\eps = \myunderbar{U} + \eps \, U^0_\eps \,, \qquad U^0_\eps = {}^t (n_{e \eps}^0, n_{i \eps}^0, {}^tv_{e \eps}^0,{}^tv_{i \eps}^0,{}^tE^0_\eps, {}^tB^0_\eps) \,.
\]
The expression $ U^0_\eps (x) $ may be just $ U^0_0 (x) $ or a series in powers of $ \eps $ like
\[
U^0_\eps (x) = U^0_0 (x) + \eps \, U^0_1 (x) + \eps^2 \, U_2^0 (x) + \cdots + \cO (\eps^N) \, , \qquad n \in \mathbb N^* \, .
\]
The terms $ U_j^0 $ with $ j \geq 1 $ may even contain some rapid oscillations. Still, we assume that there exists an initial data $U_0^0(x) $ such that
\begin{equation}
  \label{asst0}
  \ \ U^0_0 = {}^t (n_{e 0}^0, n^0_{i 0}, {}^tv_{e 0}^0,{}^tv_{i 0}^0,{}^tE^0_0,{}^tB^0_0) \in H^s \, ,  \qquad \lim_{\eps \rightarrow 0+} \ \| U^0_\eps - U^0_0 \|_{H^s} = 0 \,.
\end{equation}
In addition, according to \eqref{Gausslaw}, we impose
\begin{equation}
  \label{Gausssld}
 \nabla \cdot B^0_\eps = 0 \,, \qquad \nabla \cdot E^0_\eps + n_{e\eps}^0 - n_{i\eps}^0 = 0\,.
\end{equation}
In particular, this implies that
\begin{equation}
  \label{Gausssldencore}
 \nabla \cdot B^0_0 = 0 \,, \qquad \nabla \cdot E^0_0 + n_{e0}^0 - n_{i0}^0 = 0\,.
\end{equation}
The two constraints inside \eqref{Gausssld} are propagated by the Maxwell--Amp\`ere equation and by the  Maxwell--Faraday equation, i.e., respectively the fifth and sixth equations of \eqref{TFEM}. The system \eqref{TFEM} is symmetrizable. The Cauchy problem \eqref{TFEM}-\eqref{Gausslaw}-\eqref{barotrop} associated with initial data in $ H^s$ is therefore locally well-posed in the whole space as well as in periodic domains. We select a (non-zero)  density $ \underline{n} \in \RR_+^* $. In so doing, for $ \eps_0 $ small enough, the vacuum can be avoided, and the lifespan $ \cT_\eps \in \RR_+^* \cup \{+\infty \} $ of the solution $ {\rm U}_\eps = \myunderbar{U} + \eps \, {U}_\eps $ is (at least) inversely proportional to the $ H^s $-norm of the small  perturbation $ \eps \, U^0_\eps $. From \eqref{asst0}, we can deduce that
\[
\sup \ \bigl \lbrace  \| U^0_\eps \|_{H^s} \, ; \, \eps \in ]0,\eps_0] \rbrace < + \infty\,.
\]    
It follows that
\[
\exists \, T \in \RR_+^* \, ;  \quad \cT_\eps \geq  T /\eps\,, \quad \forall \eps \in ]0,\eps_0 ]\,.
\]
In this text, we implement periodic solutions\footnote{There are many nuances \cite{Ala08,Ga} between the whole space $\RR^3$, domains with (solid-wall) boundary conditions, and the periodic box $\TT^3$. All these situations are interesting, with specific features (various properties of dispersion, $ \cdots $). We opt here for the periodic discussion because it facilitates the Fourier analysis.} (say $ 2 \pi $-periodic solutions), that is a spatial variable $ x $ in the torus $ \TT^3 $ with $ \TT := \RR/ (2\pi \ZZ) $. We say that $ \tau $ is the {\it fast} time variable, whereas $ t := \eps \, \tau \in [0,T] $ is the {\it slow} time variable. The leading behavior of solutions on the long time interval $ [0,T /\eps]$ may be detected by looking at
\begin{equation}
  \label{defUepst}
\begin{array}{rl}
U_\eps (t,x) \! \! \! & = {}^t (n_{e\eps},n_{i \eps},{}^tv_{e \eps},{}^tv_{i\eps},{}^tE_\eps,{}^tB_\eps)(t,x) \smallskip \\
\ & := \eps^{-1} \ ({\rm U}_\eps - \myunderbar{U} ) (\eps^{-1} \, t,x)\, .
\end{array}
\end{equation}
Our aim is to describe on $ [0,T] $ the asymptotic behavior of the family $ \{ U_\eps (t,x) \}_\eps $ when $ \eps $ goes to zero. The ratios of the velocities to the sound speeds are given by
\[
  {\rm v}_{s \eps} / \sqrt{p'_s ({\rm n}_{s \eps})} =
  \eps \ v_{s\eps} /\sqrt{p'_s (\underline{n})} + \mathcal{O}(\eps^2 ) = \mathcal{O}(\eps ) \,, \qquad s \in \{e,i \}\,.
\]
So the discussion falls under the scope of {\it two-fluid low Mach number limits}. This topic stems from pioneer contributions of Klainerman--Majda \cite{KM, KM82} and Kreiss \cite{K}. Then, it was  developed concerning  compressible Euler equations (isentropic or not), see \cite{LM,MS,S07} or the survey articles  \cite{Ala08,Ga}. By extension, it has more  recently been investigated regarding ideal  magnetohydrodynamics, as for  instance in  \cite{BC2,JJLX,JSX19}.

Among existing approaches, the {\it filtering unitary group method} initiated by S. Schochet is particularly efficient since it works in a broad context including \eqref{TFEM}-\eqref{Gausslaw}-\eqref{barotrop}. The general theory relies on a reformulation of the equations into a penalized system like
\begin{equation}
  \label{HSeps}
  \part_t U_\eps - \frac{1}{\eps} \ \cL(\myunderbar{U}) \, U_\eps = A(\eps,\myunderbar{U},U_\eps,D_x) \, U_\eps + f(\eps,\myunderbar{U},U_\eps) \,.
\end{equation}
In \eqref{HSeps}, the linear operator $ \cL = \cL(\myunderbar{U})$ is a skew-adjoint operator given by a differential operator of order $ 1 $ with constant coefficients. Moreover,
\[
A(\eps,\myunderbar{U},U_\eps,D_x) := \sum_{j=1}^d A_j(\eps,\myunderbar{U},U_\eps) \ \part_{x_j}\, , \qquad A_j = {}^t A_j \in \mathcal M_{N\times N} (\RR)\,,
\]
where the integers $ d \in \NN^* $ and $ N \in \NN^* $ denote respectively the space dimension and the number of state variables (in our situation, we have $ d=3 $ and $ N = 14 $). In addition, the $ d $ functions $ A_i: [0,\eps_0] \times \RR^N \times \RR^N  \rightarrow \cM_{N \times N} (\RR)  $ and $ f: [0,\eps_0] \times  \RR^N \times \RR^N \rightarrow \RR^N $ are assumed to be smooth in their arguments. Then \cite{Met09, Sch94b, S07}, there exists a {\it profile} $ U_0^r (\tau,t,x) $, where the exponent $ r $ is for \underline{r}apid, such that
\begin{equation}
  \label{difflim}
  \forall \sigma <  s \,, \qquad \lim_{\eps \rightarrow 0 +} \| U_\eps (t,x) - U_0^r (t/\eps,t,x) \|_{\mathscr{C}([0,T];H^{\sigma})} = 0\,.
\end{equation}
The expression $ U_0^r $ can be factorized into
\[
U_0^r (\tau,t,x) = e^{\tau \cL} \, U_0 (t,x) \,, \qquad U_0 = {}^t (n_{e 0}, n_{i 0}, {}^tv_{e0},{}^tv_{i0},{}^tE_0,{}^tB_0)\,.
\]
By construction, the function $ U_0^r (\cdot,t,x) $ is almost periodic with respect to $ \tau $, and we can implement the time-averaging (mean value) operator
\[
\MM_\tau \bigl( U_0^r (\cdot,t,x) \bigr) := \lim_{\cT \rightarrow + \infty} \frac {1}{\cT}\int_0^\cT  U_0^r (\tau,t,x) \ d \tau \,.
\]
Observe that $ U_0 (t,x) $ can be seen as the weak $ L^2 ([0,T]\times \RR^3)$-limit\footnote{In the general case, the strong $ L^2 $-limit of $U_\varepsilon$ as $\varepsilon \rightarrow 0 $ may not exist.} of the family $ \{ U_\eps \}_\eps $ when $ \eps $ goes to zero. It must satisfy the following {\it modulation equation} (called FLM for Fast Limit Model)
\begin{equation}
  \label{modu0}
  \qquad \ \part_t U_0 = \MM_\tau \Bigl \lbrack e^{- \tau \, \cL } \, \bigl( A(0,\myunderbar{U}, e^{\tau \, \cL } \, U_0,D_x)
  \, e^{\tau \, \cL } \, U_0 \bigr) + e^{- \tau \, \cL } f(0,\myunderbar{U}, e^{\tau \, \cL } \, U_0) \Bigr \rbrack \,,
\end{equation}
associated with the initial data
\begin{equation}
  \label{modu0ini0}
U_0 (0,\cdot) = U^0_0 \,.
\end{equation}
We refer to Appendix~\ref{appendix:ME} for a derivation and a justification of \eqref{difflim}-\eqref{modu0}-\eqref{modu0ini0} from \eqref{HSeps}. The Cauchy problem \eqref{modu0}-\eqref{modu0ini0} is locally well-posed in $ H^s $. Since $ \cL = - \cL^* $, we can find projectors $ \PP $ and $ \QQ $ such that
\[
\PP = \PP^* = \PP^2\,, \quad  \QQ = \QQ^* = \QQ^2\,, \quad  \PP + \QQ = \text{Id}\,,
\quad  \cL \, \PP = \PP \, \cL = 0 \,, \quad  \cL \, \QQ = \QQ \, \cL = \QQ \, \cL \, \QQ \,.
\]
In other words, $\PP$ is the orthogonal projector onto the kernel ${\rm Ker}\, \cL$ of $\cL$, while  $\QQ$ is the orthogonal projector onto the image ${\rm Im} \, \cL$ of $\cL$. By construction, the operator $ \QQ \, \cL \, \QQ $ is skew-adjoint with non-zero spectrum. This means that $ \QQ \, \cL \, \QQ $ is invertible, and that
\begin{equation}
  \label{ecc}
  e^{\tau \, \cL } \, \QQ \, U_0 (t) = e^{\tau \, \QQ \, \cL \,  \QQ} \, \QQ \,U_0(t) \,, \qquad \MM_\tau \bigl \lbrack e^{\tau \, \QQ \, \cL \,  \QQ} \, \QQ \, U_0(t) \bigr \rbrack = 0 \,.
\end{equation}
The analysis of the equation \eqref{modu0} leads to a distinction between {\it slow} configurations (relying only on $ t $) and {\it fast} ones  (involving both $ \tau $ and $ t $). The slow case is when the profile $ U_0^r (\tau,t,x) $ does not depend on $ \tau $, so that $ U_0^r \equiv U_0 $. Observe that $ U_0^r \equiv U_0 $ if and only if $ U_0 = \PP \, U_0 $, which implies $e^{ \tau \, \cL }\,U_0 = \PP \, U_0$. Then, using \eqref{ecc}, it is easy to see that \eqref{modu0} reduces to the following Slow Limit Model (SLM)
\begin{equation}
  \label{modu0prepa}
\part_t (\PP \, U_0) = \PP \, \bigl \lbrack A(0,\myunderbar{U},\PP \, U_0,D_x) \, \PP \, U_0 \bigr \rbrack + \PP \, f(0,\myunderbar{U},\PP \,U_0) \,.
\end{equation}
The initial data $ U_0^0 $ is said to be {\it prepared} when $ U_0^0 = \PP \, U_0^0 $, or equivalently when $U_0^0$ satisfies $\cL \,U_0^0 =0$. More generally, the family $ \{ U^0_\eps \}_\eps $ is called prepared when $ {\part_t U_\eps}_{\mid t = 0 } = \mathcal{O}(1) $ holds true in $ H^{s-1} $. It is called {\it unprepared} when $ {\part_t U_\eps}_{\mid t = 0 } = \mathcal{O}(1/\eps)  $.

The importance of SLM, that is of \eqref{modu0prepa}, is not limited to the periodic context $ \TT^3 $. Much to the contrary, it is also  considered as the generic model when dealing in the whole space $ \RR^3 $ with  localized unprepared data (in $ H^s $ with $ s > 5/2 $). This is due to the dispersive properties of the unitary group associated with the penalized (linear) operator\footnote{Historically \cite{KM82,Uka86}, for compressible Euler equations, the relevance of SLM was proved through dispersive estimates on the wave equation. This argument was improved or adapted in \cite{DG99,Mas01} through Strichartz estimates \cite{KT98}. Alternatively \cite{MS00,MS}, microlocal defect measures could be exploited to show the convergence of EMTF towards SLM.}.

As for incompressible limits \cite{Ga,KM,LM,MS,S07}, $ U_0^0 $ is prepared on condition that
\begin{equation}
  \label{incompre}
 \nabla \cdot v^0_{e0} = 0 \,, \qquad \nabla \cdot v^0_{i0} = 0 \,.
\end{equation}
But, as will be seen, this also includes a number of conditions on $ (n_{e 0}^0, n_{i 0}^0,{}^t E^0_0,{}^t B^0_0 ) $ which complicate the analysis. In the case of prepared data, the solution to \eqref{modu0} can be recovered by just solving \eqref{modu0prepa}. Based on previous achievements\footnote{The derivation of slow and fast modulation equations and their justification hold true for quite general  symmetrizable systems, including \eqref{TFEM}-\eqref{Gausslaw}-\eqref{barotrop}, see  Appendix~\ref{appendix:ME}. There are assumptions required on the initial data, which are verified due to \eqref{asst0}.} \cite{Ala08,KM,K,Met09,Rau12,Sch94b,S07}, the above description does apply when dealing with \eqref{TFEM}-\eqref{Gausslaw}-\eqref{barotrop}. There are however somewhat unconventional subtleties coming from  Gauss's laws. Indeed, the constraints inside \eqref{Gausslaw} lead to complications. Instead of working directly with SLM, we will have to first introduce the notion of Effective Slow Limit Model (ESLM); and we will also have to substitute $ \PP $ for some (more restrictive) \underline{e}ffective orthogonal projector $ \PP_{\! e} $. The condition $ U_0 = \PP_{\! e} \, U_0 $ is equivalent to \eqref{incompre} together with\footnote{Be careful to avoid the confusion between $ \underline n $ and $ \overline n $. The constant $ \overline n $ is distinct from $ \underline n $. It is related to the behavior of the perturbations $ n^0_{s\eps}$ of $ \underline n $.}
\begin{equation}\label{overlineeta}
 \exists \, \overline{n}\in \RR_+^* \, ; \qquad Z \, n_{i0}^0=n_{e0}^0=Z \, \overline{n}\, ,
\end{equation}
as well as
  \begin{equation}
    \label{prepared+gauss0}
    E_0^0=0\,, \qquad B_0^0 = Z \, \underline{n} \, \nabla \times \Delta^{-1}( v_{e0}^0- v_{i0}^0)\,.
  \end{equation}
From now on, we will use \eqref{modu0} and  \eqref{modu0prepa} to designate respectively the FL and SL Models issued  from \eqref{TFEM}-\eqref{Gausslaw}-\eqref{barotrop}. These (somewhat abstract) modulation equations provide information on a universal behavior in plasmas since, starting from small perturbations of a constant background,  after a certain period of time, the solutions behave like $ \myunderbar{U} + \eps \, U_0^r + o(\eps) $. Alternatively, for well-adjusted plasma parameters, resorting to this ansatz is pertinent.
Now, it is interesting to put in perspective the foregoing with known asymptotic results about conducting two-fluids, which can mainly be divided into two categories:
\begin{itemize}
\item {\it Generalized irrotational flows}. This is when
  \begin{equation}
    \label{Pausadercdt}
    \updelta \, \nabla \times {\rm v}_e = - \nabla \times {\rm v}_i = {\rm B} \,, \qquad \updelta :=Z \, m_e / m_i \, ,
  \end{equation}
  where $ \updelta / Z = (m_i/m_e)^{-1} $ reaches its maximum for protons (with  $ \delta \sim 10^{-3} $). The two  relations inside \eqref{Pausadercdt} are preserved over time. In  \cite{GIP}, under rarefied ($ \underline{n} = 0 $) and whole space ($ x \in \RR^3 $) assumptions, remarkable results of global existence ($ \cT_\eps = + \infty $) and stability are  achieved. The   solutions decay to zero in the sup-norm as the time variables $ \tau $ or $ t $ tend to $ + \infty $. The situation in \cite{GIP} differs from ours since we work mainly with $ \underline{n} \in \RR_+^* $ and $ x \in \TT^3 $. Still, in the continuity of \cite{GIP}, we explain in Paragraph \ref{subs:zerodensity} what happens in our framework when $ \underline{n} = 0 $. We  also investigate in Paragraph \ref{subs:irroflows} the consequences of imposing \eqref{Pausadercdt}. We will see that when $ \underline{n} \not = 0 $ and under \eqref{Pausadercdt}, the solution $ U_0 $ to the SL Model is  constant, implying that $ {\rm U}_\eps $ converges to this constant at a rate like $ o(\eps) $ in a time of order $ 1/\eps $. This confirms that \eqref{Pausadercdt} triggers special cancellations. By contrast, in the absence of \eqref{Pausadercdt}, the dynamics of $ U_0 $ is highly non trivial. \\
\item {\it Almost single fluid descriptions}. This is when $ U_0^0 $ is adjusted to satisfy a reinforced condition of quasineutrality and when the two initial velocities are (almost) the same. This means that
  \begin{equation}
    \label{LiPuW}
    n_{e\eps} = Z \, n_{i\eps} + \mathcal{O}(\eps) \, , \qquad v_{e\eps} = v_{i\eps} + \mathcal{O}(\eps)\,  .
  \end{equation}
  This is supplemented by the smallness of the electromagnetic field, namely
  \begin{equation}
    \label{LiPuWeb}
    \nabla \times E_\eps = \mathcal{O}(\eps) \, , \qquad B_\eps = \mathcal{O}(\eps) \, .
  \end{equation}
  Reference for this approach can be made to \cite{LPX,LPW,PL} and the numerous articles which are cited therein. Then, the characteristics of the two species become hardly distinguishable. They can be approximately described by one density (say $ n_{i0} $) and one velocity (say $ v_{i0} $). By adjusting conveniently the plasma parameters, such regimes are accessible in different ways: for instance under quasi-neutral, non-relativistic or  zero-electron-mass assumptions. The limit equations for $ n_{e0} \sim Z \, n_{i0} $ and $ v_{e0} \sim v_{i0} $ are most of the time  designed on the basis of single-fluid models: Euler equations, MHD or e-MHD. In the presence of dissipation,  a preliminary attempt towards establishing Hall  effects was made recently in \cite{PXZ}.
\end{itemize}

Compare the three groups of conditions contained in \eqref{incompre}-\eqref{overlineeta}-\eqref{prepared+gauss0} on the one hand, in \eqref{Pausadercdt} on the other hand, and lastly inside \eqref{LiPuW}-\eqref{LiPuWeb},  to see that they are all quite  distinct. The irrotational condition  \eqref{Pausadercdt} concerns the whole solutions whereas \eqref{incompre}-\eqref{overlineeta}-\eqref{prepared+gauss0} and \eqref{LiPuW}-\eqref{LiPuWeb} are only about dominant terms. The assumptions inside \eqref{LiPuW}-\eqref{LiPuWeb} are particularly restrictive because they enforce the two velocities $ v_{e0} $ and $ v_{i0} $ to be equal, and they  mean that $ B_0 = 0 $.

Now, the filtering theory does not require \eqref{Pausadercdt} or \eqref{LiPuW}-\eqref{LiPuWeb}. In the prepared case, it replaces \eqref{Pausadercdt} and \eqref{LiPuW}-\eqref{LiPuWeb} by other (less constraining) conditions. In the unprepared case, it even allows to get rid of all restrictions, except of course those coming from \eqref{Gausslaw}. This can be done by incorporating through FLM the influence of oscillations, while the preceding articles \cite{LPW,PL,PXZ} and connected papers are often limited to ``non-oscillating'' Hilbert expansions (which do not include the presence of rapid variations).

The reasons why the powerful filtering unitary group method (and its FLM and SLM outcomes) has not yet been exploited in the context of EMTF are probably threefold:

\begin{itemize}
\item[i)] The first is, no doubt, the complexity of the calculations\footnote{While incompressible limits of single-fluid equations have been intensively studied \cite{Ala08,Ga,S07}, to our knowledge, the two-fluid framework has never been investigated in this way. This is  probably due to the underlying significant technical complications.}. As will be seen, the asymptotic study of the diverse interactions between the two velocities and the electromagnetic field requires a careful Fourier analysis (through Fourier multipliers or pseudo-differential tools).
\item[ii)] The second are ways of thinking which may be commonly turned to single fluid conceptions. By contrast, by breaking   \eqref{Pausadercdt} and \eqref{LiPuW}-\eqref{LiPuWeb}, numerous two-fluid effects can be restored. The FLM and SLM models issued from EMTF help to understand some of their important impacts.
\item[iii)] The third, which is related to the preceding explanation, is the (apparent) absence of a target model (well identified in the mathematical community) that could serve to motivate the study. The EMTF system is the source of very rich and complex plasma dynamics. The fact remains that, under other scalings and  hypotheses,  EMTF can lead to a panel of (mostly dispersive) models like Korteweg--de Vries, Kadomtsev--Petviashvili, Zakharov \cite{DLT}, nonlinear Schr\"odinger \cite{Met09} (and so on, see references in \cite{Met09}) that have received a great deal of attention (due perhaps to their more accessible structure). In this list, XMHD is up to now absent.
\end{itemize}

  The three important points i), ii) and  iii) are precisely what  we would like to clarify. We claim that, in the prepared case\footnote{This result should hold true also in the whole space $ \RR^3 $ with  localized unprepared data, even if this direction of research is not developed here..}, limit equations do exist, and that they are completely equivalent to incompressible XMHD\footnote{XMHD is the  set of equations \eqref{lienBB*2}-\eqref{divfreeini}-\eqref{syssimpli} that will subsequently be introduced. We refer the reader to  \cite{BC1} for a study of the Cauchy problem related  to XMHD.}, which has become quite a topical subject in plasma physics, see for instance \cite{AKY,AL,ALM,AGMDG,DML16,GTAM,GP04,KLMLW,KMo,LMMbis,Lu,MLM,SMA}.

As explained in Appendix \ref{s:appendix}, XMHD can be  addressed from two connected perspectives: fluid mechanics\footnote{L\"ust \cite{Lu}-(in 1959) was historically  the first to derive XMHD, already by starting from a two-fluid model, and then by enforcing neutrality and by performing a truncation of the expansion in powers of the electron mass $ m_e $. In our context, the plasma is almost neutral, due to \eqref{neutral}, and there is no truncation implying the smallness of $ m_e $. Instead, we  exploit a hierarchy in powers of the amplitude $ \eps $ of the perturbation.} (Subsection  \ref{sub:fluid}) and Hamiltonian mechanics\footnote{XMHD can be obtained from Hamiltonian's principle. It is compatible with Hamiltonian formalism \cite{AKY,DML16}. Then, the physics lies in the choice of the total energy \cite{KMo}, which is inspired from two-fluid considerations.} (Subsection \ref{sub:hamiltonian}). The issues raised by XMHD are at the cutting edge of mathematical progress in the field of  magnetohydrodynamics (Subsection  \ref{sub:math}). XMHD sheds new light on major scientific challenges: (two-fluid plasma) turbulence, Hall impacts and inertial effects. XMHD is also likely to offer promising developments in  collisionless magnetic reconnection.

Let us recall the content of XMHD. To this end, we have to introduce the dimensionless parameters  $\myunderbar{\rho}$,  $d_e$, and $d_i$ which stand respectively for the normalized total mass density, the normalized electron skin depth (resurgence of electronic inertial effects), and the normalized ion skin depth (resurgence of ionic inertial effects or Hall effects). These parameters are given in terms of the dimensionless density $\underline{n}$,  the dimensionless electronic mass $m_e$, the dimensionless ionic mass $m_i$, and the charge number $Z$ according to the formulas
\begin{equation}
  \label{eqn:cteXMHD}
\qquad  \myunderbar{\rho}:=\underline{n} \, (m_e+ m_i) \, , \qquad d_e :=\sqrt{\updelta}\  \frac{m_i}{Z} \, , \qquad d_i := (1-\updelta) \, \frac{m_i}{Z} \,,
\end{equation}
with $\updelta=Zm_e/m_i$. In XMHD, the unknown is $ (u_0,B^*_0) \in \RR^3 \times \RR^3 $. The part $ u_0 $ is akin to the center of mass velocity, which is
\begin{equation}\label{centermasselocity}
 u_0 := \frac{m_e \, v_{e0} + m_i \, v_{i0}}{m_e + m_i} \, , \qquad u_0^0 := \frac{m_e \, v_{e0}^0 + m_i \, v_{i0}^0}{m_e + m_i} \, .
\end{equation}
The part $ B^*_0 $ is connected to the usual magnetic field $ B_0 $ through the constitutive relation
\begin{equation}
  \label{lienBB*2}
  B_0 = \big(\id - {\underline b} \,\Delta\big)^{-1} B^*_0 \, , \qquad {\underline b} := d_e^2/\myunderbar{\rho} \, .
\end{equation}
All fields are divergence free
\begin{equation}
  \label{divfreeini}
\nabla \cdot u_0 = 0 , \qquad \nabla \cdot B^*_0 = 0 , \qquad \nabla \cdot B_0 = 0 \,,
\end{equation}
and they evolve according to
\begin{equation}
  \label{syssimpli}
  \left \lbrace \begin{array}{l}
    \displaystyle \part_t u_0 + (u_0\cdot \nabla ) u_0 + \nabla p_0 +  \myunderbar{\rho}^{-1}B^*_0 \times (\nabla \times B_0) = 0 \,, \medskip \\
    \displaystyle \part_t B^*_0 + \nabla \times \bigl( B^*_0 \times (u_0 - (d_i/ \myunderbar{\rho}) \,\nabla \times B_0) \bigr) \smallskip\\
    \qquad \  + \, (d_e^2/\myunderbar{\rho})\ \nabla \times \bigl( (\nabla \times u_0)
    \times (\nabla \times B_0 ) \bigr) = 0 \,.
  \end{array} \right.
\end{equation}
The system \eqref{lienBB*2}-\eqref{divfreeini}-\eqref{syssimpli} is closed, and well-posed \cite{BC1}. Now, the complex interplay between \eqref{modu0prepa} and \eqref{syssimpli} is elucidated in the next  statement.

\begin{theo}[Equivalence between SLM and XMHD]
  \label{maintheo}
Assume that $ U^0_0 $ is prepared in the sense that the three conditions \eqref{incompre}, \eqref{overlineeta} and \eqref{prepared+gauss0} are verified.  Then, the asymptotic behavior  of the family $ \{ U_\eps \}_\eps $ is controlled when $ \eps \rightarrow 0 $ as in \eqref{difflim} with a profile $ U^r_0 \equiv U_0 $ not depending on the fast time variable $ \tau $.

More precisely, we find that $ U_0 = \PP \, U_0 = \PP_{\! e} \, U_0 $ is a solution to the Slow Limit Model \eqref{modu0prepa} with initial data $ U_0^0 $. With $ \overline{n} $ as in \eqref{overlineeta}, the expression $ U_0 $ is such that
\begin{equation}
  \label{whichissuch}
  U_0 (t,x)= {}^t \bigl(Z \, \overline{n},\overline{n}, {}^t v_{e0}(t,x), {}^t v_{i0}(t,x), 0 , {}^t B_0(t,x)\bigr) \, .
\end{equation}
Moreover, let $ (u_0,B_0^*) $ be the solution to the XMHD equations \eqref{lienBB*2}-\eqref{divfreeini}-\eqref{syssimpli} with initial data
\[
(u_0,B_0^*)_{\mid t=0} = \bigl( u_0^0 , (\id - \Delta)^{-1} B_0^0 \bigr) \,.
\]
Then, we have
\begin{equation}
  \label{importantrelation}
  \left( \begin{array}{c}
    v_{e0} \\
    v_{i0} \\
    B_0
  \end{array} \right) = \mathfrak{I} \left( \begin{array}{c}u_0 \\
    B^*_0
  \end{array} \right) , \qquad \mathfrak{I} := \left( \begin{array}{cc}
    \id \ \ & - \, ({\underline b}/m_e) \, \nabla \times \big(\id - {\underline b} \,\Delta\big)^{-1}\\
    \id \ \ & \ \ \, ({\underline b}/m_i) \, \,\nabla \times \big(\id - {\underline b} \,\Delta\big)^{-1}\\
    0 \ \ & \big(\id - {\underline b} \,\Delta\big)^{-1}
  \end{array} \right)\, .
\end{equation}
Conversely, we can deduce a solution $ (u_0,B_0^*) $ to XMHD from any solution $ (v_{e0},v_{i0},B_0) $ to SLM by using \eqref{centermasselocity} and \eqref{lienBB*2}.
\end{theo}
XMHD and SLM equations are therefore in one-to-one correspondance. But they involve distinct representation modes. The corresponding phase spaces $ L^2(\RR^3;\RR^6) $ and $ L^2(\RR^3;\RR^{14}) $ are different, with numbers\footnote{The count of $ 14 $ in $ U_0 \in L^2(\RR^3;\RR^{14}) $ does not take into account the polarization condition $ U_0 = \PP_{\! e} U_0 $.} of unknowns equal to $ 6 $ and $ 14 $. They are connected by the pseudo-differential (or Fourier multiplier) embedding $ \mathfrak{I} $. Observe in particular that the  regularities of $ U_0 $ and $ B^*_0 $  are not the same. The XMHD and SLM interpretations have each its own advantages (and disadvantages):
\begin{itemize}
\item XMHD is easily accessible with its reduced set of unknowns $ (u_0,B_0^*) \in \RR^3 \times \RR^3 $ (instead of $ U_0 \in \RR^{14}$) and its differential formulation. It allows to get rid of the complicated operations induced by the constraint $ U_0 = \PP_{\! e} \, U_0 $. It also gives through $ (u_0 , B^*_0) $ a direct access to the principal features of the flow. The implementation in XMHD of a single velocity $ u_0 $, which as indicated in \eqref{centermasselocity} can be interpreted as the velocity of the center of mass, must not  obscure the underlying two-fluid aspects. The reader must keep in mind that $ u_0 $ strongly differ from $ v_{e0} $ and $ v_{i0} $. Again, one of the roles of  $ \mathfrak{I} $ is to make apparent this distinction.
\smallskip
\item The SLM viewpoint relies on a genuine two-fluid vision since the components of $ U_0 $ can simply be connected with those of $ U_\eps $. Once $ B_0^* \not \equiv 0 $, it is possible to differentiate the electon velocity $ v_{e0} $ from the ion velocity $ v_{i0} $, with generically $ v_{e0} \not = v_{i0} $,  as clearly the injection $\mathfrak I$ shows. The SLM approach is  intrinsically related to EMTF, and it leads  to well-posedness\footnote{The Cauchy problem is by construction well-posed at the level of the (symmetric) SL Model. By contrast, it is
quite delicate to solve at the level of XMHD. This is achieved in \cite{BC1} through adequate pseudo-differential transformations; this is obtained in \cite{BC1} by introducing a potential formulation which mimics some features of SLM (without $ \PP $ but with technicalities induced by the constitutive relation).}. That is also a necessary prelude before including the additional impacts of fast waves.
\end{itemize}

Still in a periodic box but now for  unprepared data, the Fast Limit Model appears as a generalization of SLM (and therefore of XMHD) that allows to better understand the impact of fast waves on the XMHD  slower dynamics. Basically, waves with different speeds may interact by nonlinearity to produce new waves propagating at other velocities.
This is the phenomenon of resonance \cite{Dan02,Ga98,Gre97,Mas01,MS00,MS03,Sch94b}.

The plan of the paper is as follows. In Section \ref{SlowLimitModel}, we detail the concrete content of SLM. In Appendix A, we present  the theoretical developments that are brought into play in order to justify the approximation \eqref{difflim}. In Appendix B, which lies at the interface between physics and mathematics, we provide a comprehensive overview on XMHD.

\section{The Slow Limit Model}
\label{SlowLimitModel}
This section is devoted to the proof of Theorem~\ref{maintheo}, which is divided into four  steps. The first stage (Subsection \ref{subs:reformulation}) is to interpret the EMTF system (written in terms of $ U_\eps $) as a quasilinear symmetric penalized system (on a new state variable $ \cU_{\varepsilon} $) involving a singular part $\varepsilon^{-1}\cL \, \cU_{\varepsilon} $, where   the action of $ \cL $ is  skew-adjoint. Then (Subsection \ref{subs:impact}), we observe that the filtering unitary group method of Schochet \cite{S07} can be applied.

However, adaptations are needed in order to incorporate the supplementary constraints induced by Gauss's laws. This leads to the notion of {\it \underline{e}ffective} orthogonal projector $ \PP_{\! e} $ (distinct from the orthogonal projector $ \PP $ onto the kernel of $ \cL $). The third operation (Subsection \ref{subs:PP}) is to identify the content of $\PP_{\! e} $. The last step (Subsection \ref{subs:end}) consists in exploiting the explicit formula obtained for $\PP_{\! e} $ to show the equivalence between  SLM and  incompressible XMHD.


\subsection{Reformulation of the EMTF system}
\label{subs:reformulation}
The starting point is the EMTF system in the conservative form  \eqref{TFEM}. Recall that,  without loss of generality, we can take $Z=1$. Since $ p_s' >0$ on $ \RR_+^* $, instead of working with $ {\rm n}_s$, we can implement
\begin{equation*}
  \label{implens}
        {\rm q}_s := g_s({\rm n}_s)\,, \qquad g_s({\rm n}_s):=\frac{1}{\sqrt{m_s}}\int_{\underline{n}}^{{\rm n}_s} \frac{\sqrt{p_s'(\theta)}}{\theta}\ d\theta\,,
        \qquad s\in\{e,i\}\,.
\end{equation*}
This change of density variables leads to the introduction of two scalar functions $ a_s : \RR \rightarrow \RR_+ $ given by
\begin{equation*}
  \label{changedensity}
  \qquad \qquad \ a_s({\rm q}_s) :=g_s^{-1}({\rm q}_s) \, g_s'\circ g_s^{-1}({\rm q}_s)=g_s'({\rm n}_s) \, {\rm n}_s= \sqrt{\frac{p_s'({\rm n}_s)}{m_s}}\,, \qquad s\in\{e,i\}\,,
\end{equation*}
whose values at the image  $ g_s (\underline{n}) = 0 $ of the equilibrium position $ \underline{n} $ are such that
\begin{equation}
  \label{defpaprime}
  \underline{a}_s := a_s(0) =  \sqrt{\frac{\underline{p}_s'}{m_s}}\,, \qquad \underline{p}_s' := p_s'(\underline{n}) , \qquad s\in\{e,i\}\,.
\end{equation}
The interest of using $ {\rm q}_s $ is to give access to some equivalent  symmetric quasilinear version of EMTF, namely
\begin{equation}\label{TFEMS}
  \left\{
  \begin{array}{l}
    \part_{\tau} {\rm q}_e +  {\rm v}_e\cdot \nabla  {\rm q}_e + a_e({\rm q}_e) \nabla \cdot {\rm v}_e = 0 \,, \smallskip \\
    \part_{\tau} {\rm q}_i +  \,{\rm v}_i\cdot \nabla {\rm q}_i \,+ a_i({\rm q}_i) \nabla \cdot {\rm v}_i \,= 0 \,, \smallskip \\
   \part_{\tau} {\rm v}_e +  {\rm v}_e \cdot \nabla {\rm v}_e  + a_e({\rm q}_e) \nabla {\rm q}_e 
    = - ({\rm E} + {\rm v}_e \times {\rm B})/m_e \,, \smallskip \\
    \part_{\tau} {\rm v}_i + \, {\rm v}_i \cdot \nabla {\rm v}_i  \, + a_i({\rm q}_i) \nabla {\rm q}_i 
    \, \, = \ \ ({\rm E} + {\rm v}_i \times {\rm B})/m_i \,, \smallskip \\
    \displaystyle \part_{\tau} {\rm E} \, - \, \nabla \times {\rm B}  =g_e^{-1}({\rm q}_e) {\rm v}_e - g_i^{-1}({\rm q}_i) {\rm v}_i  \,, \smallskip \\
    \part_{\tau}  {\rm B} \, + \, \nabla \times {\rm E} = 0 \,, 
  \end{array}
  \right.
\end{equation}
together with
\begin{equation}
  \label{GausslawS}
  \nabla \cdot {\rm B}=0\,, \qquad\nabla \cdot {\rm E}=g_i^{-1}({\rm q}_i) - g_e^{-1}({\rm q}_e)\,.
\end{equation}
In this new setting, we can resume the procedure of Section~\ref{s:intro}. To avoid the multiplication of notations, without risk of confusion, we keep ${\rm U}$, ${\rm U}_\varepsilon$, $\myunderbar{U}$ and $U_\varepsilon$. The flat equilibrium for \eqref{TFEMS}-\eqref{GausslawS} corresponding to \eqref{neutral} is just $ \myunderbar{U} = 0 $. Let us define the two scalar coefficients
\begin{equation}
  \label{defabbar}
  \underline{b}:= \frac{d_e^2}{\myunderbar{\rho}} = \bigg(\frac{\underline{n}}{m_e} + \frac{\underline{n}}{m_i}\bigg)^{-1} \, , \qquad \underline{c}:=\Big(\sqrt{{\underline{p}_e'}/{\underline{p}_i'}}+ \sqrt{{\underline{p}_i'}/{\underline{p}_e'}}\Big)^{-1}\,.
\end{equation}
Equivalently to what has been done in the introductory  Section~\ref{s:intro}, the objective is to seek solutions ${\rm U}={}^t({\rm q}_e,{\rm q}_i, {}^t {\rm v}_e ,{}^t {\rm v}_i,{}^t {\rm E},{}^t {\rm B}) $ to \eqref{TFEMS}-\eqref{GausslawS} having the form
\begin{equation}
  \label{linearalo0}
        {\rm U}_\varepsilon(\tau,x) = \varepsilon \,  U_\varepsilon(\varepsilon \tau,x) \, , \qquad  U_\varepsilon = {}^t(q_{e \eps},q_{i \eps},{}^t v_{e \eps},{}^t v_{i \eps},{}^t E_\eps,{}^t B_\eps) \, .
\end{equation}
Given $ h \in \mathscr C^\infty (\RR;\RR) $ and a fixed position $ \underline{q} $, we denote by $ \mathscr R (h,\underline{q}): [0,\eps_0] \times \RR \rightarrow \RR $
the smooth function determined by
\[
\mathscr R (h,\underline{q}) (\eps,q) := \frac{1}{\varepsilon}\, \big(h(\underline{q} +\varepsilon q) - h(\underline{q})\big) = \int_0^1 h'(\underline{q} +\varepsilon q \, r) \, q \ dr \,.
\]
When $ \underline{q} = 0 $, we just use $ \myunderbar{\mathscr R} (h) := \mathscr R (h,0) $. Retain
\begin{equation}
  \label{mathscrR}
  \qquad \mathscr R (h,\underline{q}) (\eps,q) = h'(\underline{q}) \, q + \cO(\eps) \, , \qquad \mathscr R (h,\underline{q}) (0,q) = h'(\underline{q}) \, q \, .
\end{equation}
In place of \eqref{TFEMS}, we find
\[
  \left\{
  \begin{array}{l}
    \displaystyle \part_{t} q_{e\varepsilon} +  v_{e\varepsilon}\cdot \nabla  q_{e\varepsilon}
    + \myunderbar{\mathscr R} (a_e) (\eps, q_{e\varepsilon}) \,  \nabla \cdot v_{e\varepsilon}
    + \frac{\underline{a}_e}{\varepsilon} \, \nabla \cdot v_{e\varepsilon} = 0 \,, \medskip \\
    \displaystyle \part_{t} q_{i\varepsilon} \, +  v_{i\varepsilon}\cdot \nabla  q_{i\varepsilon}
    \, + \myunderbar{\mathscr R} (a_i) (\eps, q_{i\varepsilon}) \  \nabla \cdot v_{i\varepsilon} \,
    + \frac{\underline{a}_i}{\varepsilon} \, \nabla \cdot v_{i\varepsilon} \, = 0\,, \medskip \\
    \displaystyle \part_{t} v_{e\varepsilon} +  v_{e\varepsilon} \cdot \nabla v_{e\varepsilon}
    + \myunderbar{\mathscr R} (a_e) (\eps, q_{e\varepsilon}) \,  \nabla q_{e\varepsilon}
    + \frac{1}{\varepsilon}\, \bigg(\underline{a}_e \,  \nabla q_{e\varepsilon} + \frac{E_{\varepsilon}}{m_e}\bigg)= - \frac{v_{e\varepsilon}}{m_e}  \times B_{\varepsilon} \,, \medskip \\
    \displaystyle \part_{t} v_{i\varepsilon} +  v_{i\varepsilon} \cdot \nabla v_{i\varepsilon} \,
    + \myunderbar{\mathscr R} (a_i) (\eps, q_{i\varepsilon}) \ \nabla q_{i\varepsilon} \,
    + \frac{1}{\varepsilon}\, \bigg(\underline{a}_i \,  \nabla q_{i\varepsilon} - \frac{E_{\varepsilon}}{m_i}\bigg)\ = \, +  \frac{v_{i\varepsilon}}{m_i} \times B_{\varepsilon}  \,, \medskip \\
    \displaystyle \part_{t} E_{\varepsilon} +\frac{1}{\varepsilon}\, \big(\underline{n} \, (v_{i\varepsilon}-v_{e\varepsilon}) - \nabla \times B_{\varepsilon} \big) = \myunderbar{\mathscr R} (g_e^{-1}) (\eps, q_{e\varepsilon}) \,  v_{e\varepsilon} - \myunderbar{\mathscr R} (g_i^{-1}) (\eps, q_{i\varepsilon}) \,  v_{i\varepsilon} \,, \medskip \\
     \displaystyle \part_{t}  B_{\varepsilon} + \frac{1}{\varepsilon}\nabla \times E_{\varepsilon}  = 0 \, .
  \end{array}
  \right.
\]
In the linearization procedure \eqref{linearalo0}, the source term of \eqref{TFEMS} contributes with a singular part (with $ 1/\eps $ in factor), which is a linear operator. However, the action of this operator is not skew-adjoint for the standard $L^2$ scalar product. Indeed, the terms $E_{\varepsilon}/m_e$ and  $- E_{\varepsilon}/m_i$ do not match with $\underline{n}(v_{i\varepsilon}-v_{e\varepsilon})$ in a skew-adjoint way. To remedy this, we introduce the change of unknowns
\[
\cU_{\varepsilon}={}^t(\varrho_{e\varepsilon},\varrho_{i\varepsilon},{}^tu_{e\varepsilon},{}^tu_{i\varepsilon},{}^tE_{\varepsilon},{}^tB_{\varepsilon}) := \cA_0^{1/2} \,  U_{\varepsilon} = \cU_0 + \cO(\eps) \,,
\]
with
\begin{equation*}
  \cA_0:=\left( \begin{array}{cccccc}
    \underline{n} \, m_e & 0 & 0 & 0 & 0 & 0 \\
    0 & \underline{n} \, m_i & 0 & 0 & 0 & 0 \\
    0 & 0 & \underline{n}\, m_e \, {\rm I}_{3\times 3} & 0 & 0 & 0 \\
    0 & 0 & 0 & \underline{n} \, m_i \,  {\rm I}_{3\times 3} & 0 & 0 \\
    0 & 0 & 0 & 0 &  {\rm I}_{3\times 3}  & 0 \\
    0 & 0 & 0 & 0 & 0  &  {\rm I}_{3\times 3} 
  \end{array}  \right)\,.  
\end{equation*} 
Retain
\begin{equation}
  \label{passagecUU0}
  \begin{array}{rl}\cU_0 \! \! \! & = \cA_0^{1/2} \,  U_0 = {}^t(\varrho_{e0},\varrho_{i0},{}^tu_{e0},{}^tu_{i0},{}^tE_{0},{}^tB_{0}) \smallskip \\
    \ & = {}^t \Bigl(\sqrt{\underline{p}'_e/\underline n} \ n_{e0},\sqrt{\underline{p}'_i/\underline n} \ n_{i0},\sqrt{\underline n \, m_e} \ {}^t v_{e0}, \sqrt{\underline n \, m_i} \ {}^t v_{i0},{}^t E_{0},{}^t B_{0} \Bigr) \, .  \end{array}
\end{equation}
Let $e^j$ with $ j \in \{1,2,3\} $ be the $j$-th vector of the canonical basis of $ \RR^3 $. Expressed in the  variable $\cU_\varepsilon$,  our  system  becomes
\begin{equation}
  \label{TFEMS3}
  \partial_t \cU_\varepsilon - \frac{1}{\varepsilon} \, \cL \,\cU_\varepsilon =
   \sum_{j=1}^3 \cA_j(\eps,\cU_\eps) \, \part_{x_j} \cU_\eps + \cF(\eps,\cU_\eps)\,.
\end{equation}
We find that $ \displaystyle \cL := \cC_1 \, \part_{x_1} + \cC_2 \, \part_{x_2} + \cC_3 \, \part_{x_3} + \cD $ is a constant matrix-valued differential operator of order one, built with the $ 14 \times 14 $ matrix
  \begin{equation*}
  \cD:=\sqrt{\underline{n}}
    \, \left( \begin{array}{cccccc}
   0 & 0 & 0 & 0 & 0 & 0 \\
   0 &  0 & 0 & 0 & 0 & 0 \\
   0 & 0 & 0 & 0 & -{\rm I}_3/\sqrt{m_e} & 0 \\
   0 & 0 & 0 & 0  & \ \ {\rm I}_3/\sqrt{m_i} & 0 \\
    0 & 0 & {\rm I}_3/\sqrt{m_e} & - {\rm I}_3/\sqrt{m_i} & 0  & 0\\
    0 & 0 & 0 & 0 & 0  & 0
  \end{array}  \right)\,,
\end{equation*}
and with the three $ 14 \times 14 $ matrices
\begin{equation*}
  \cC_j:= - \left( \begin{array}{cccccc}
   0 & 0 & \underline{a}_e \, {}^t e^j & 0 & 0 & 0 \\
   0 &  0 & 0 &  \underline{a}_i \, {}^t e^j & 0 & 0 \\
    \underline{a}_e \, e^j & 0 & 0 & 0 & 0 & 0 \\
   0 & \underline{a}_i \, e^j & 0 & 0  & 0 & 0 \\
    0 & 0 & 0 & 0 & 0  & {}^t \cR_j\\
    0 & 0 & 0 & 0 & \cR_j  & 0
  \end{array}  \right)\,, \qquad j \in \{1,2,3 \} \,,
\end{equation*}
where the $ 3 \times 3 $ matrices $ \cR_j $ are given by
\[
\cR_1=\big(0\, \vert \, e^3 \, \vert \,  -e^2\big) \, , \qquad \cR_2=\big (-e^3 \, \vert \, 0 \, \vert \, e^1 \big ), \qquad \cR_3=\big (e^2 \, \vert \, -e^1  \, \vert \,0\big ) \,.
\]
Moreover, for $ j \in \{1,2,3 \} $, with $ \myunderbar{\mathscr R} (a_s) \equiv \myunderbar{\mathscr R} (a_s)(\eps,q_{s\eps}) $ we can compute
\[
\cA_j(\eps,\cU_\eps) :=-\left( \begin{array}{cccccc}
  \displaystyle \frac{u_{e}^j}{\sqrt{{\underline n} m_e}} & 0 & \myunderbar{\mathscr R} (a_e) \, {}^t e^j & 0 & 0 & 0 \\
  0 &  \displaystyle \frac{u_{i}^j}{\sqrt{{\underline n} m_i}} & 0 & \myunderbar{\mathscr R} (a_i) \, {}^t e^j & 0 & 0 \\
  \myunderbar{\mathscr R} (a_e) \, e^j & 0 &  \displaystyle \frac{u_{e}^j}{\sqrt{{\underline n} m_e}} \, e^j\otimes e^j & 0 & 0 & 0 \\
  0 & \myunderbar{\mathscr R} (a_i) \, e^j & 0 & \displaystyle \frac{u_{i}^j}{\sqrt{{\underline n} m_i}} \, e^j\otimes e^j & 0 & 0 \\
  0 & 0 & 0 & 0 & 0  & 0 \\
  0 & 0 & 0 & 0 & 0  & 0 
\end{array}  \right) .
\]
There remains to specify the components of $ \cF ={}^t (\cF_{\varrho_e},\cF_{\varrho_i},{}^t\cF_{u_e},{}^t\cF_{u_i},{}^t\cF_E,{}^t\cF_B) $.  We have $ \cF_{\varrho_e} = 0 $, $ \cF_{\varrho_i} = 0 $ and $ \cF_B = 0 $, whereas
\begin{equation}
  \label{remainsspecify}
  \qquad \left\{\begin{array}{l}
    \cF_{u_e} (\eps,\cU_\eps) := - \, u_{e\varepsilon}\times B_{\varepsilon}/{m_e} \, , \smallskip \\
    \cF_{u_i} (\eps,\cU_\eps) := \ \ \ \, u_{i\varepsilon}\times B_{\varepsilon}/{m_i} \, , \smallskip \\
    \displaystyle \cF_E (\eps,\cU_\eps) := \, \myunderbar{\mathscr R} (g_e^{-1})\Big(\eps, \frac{ \varrho_{e\varepsilon}}{\sqrt{{\underline n} m_e}} \Big) \ \frac{u_e}{ \sqrt{{\underline n} m_e} } - \myunderbar{\mathscr R} (
    g_i^{-1}) \Big(\eps,\frac{ \varrho_{i\varepsilon}}{\sqrt{{\underline n} m_i}} \Big) \ \frac{u_i}{ \sqrt{{\underline n} m_i} } \,  .
  \end{array}\right.
\end{equation}
In addition, from  \eqref{GausslawS}, we can deduce
\begin{equation}
  \label{GaussLawSS}
\quad \nabla \cdot B_{\varepsilon} =0\,, \qquad  \nabla \cdot E_{\varepsilon} + \myunderbar{\mathscr R} (g_e^{-1}) (\eps, q_{e\varepsilon}) - \myunderbar{\mathscr R} (g_i^{-1}) (\eps, q_{i\varepsilon}) = 0 \,.
\end{equation}
In other words, with $\underline{a}_s$ as in \eqref{defpaprime}, we must impose
\begin{equation}
  \label{defL}
\qquad  \cG \,\cU_\eps = \cO(\eps) , \qquad \cG \equiv \cG(\myunderbar{U}):=
  \left( \begin{array}{cccccc}
   0 & 0 & 0 & 0 & 0 & \nabla \cdot  \\
 \displaystyle \sqrt{\underline{n}/\underline{p}'_e} & \displaystyle - \sqrt{\underline{n}/\underline{p}'_i}  & 0 & 0 & \nabla \cdot  & 0
  \end{array}  \right)\,.
\end{equation}


\subsection{The impact of Gauss's laws}
\label{subs:impact}
As a consequence of assumption \eqref{Gausssld}, the constraints inside \eqref{GaussLawSS} are  verified at time $ t= 0 $. These conditions are propagated by the nonlinear flow. They are  transparent when performing the asymptotic  analysis of S. Schochet. This implies that the proof of Theorem~4.1 in \cite{S07} can be   repeated as it is. In the first instance, we can skip \eqref{GaussLawSS}. Let $\PP $ be the orthogonal projector onto the kernel of $ \cL $. In the case of prepared data, that is\footnote{The condition $\part_t {U_\varepsilon}_{|_{t=0}}=\mathcal{O}(1)$ is achieved if and only if $ \cL \, \cU_\eps^0 = \cO (\eps) $ that is when $ \cU_\eps^0 = \PP \, \cU_\eps^0 + \cO(\eps)$ or in view of \eqref{asst0} when $ \cU_\varepsilon^0=\PP \, \cU_0^0 + \PP o(1) + \mathcal{O}(\varepsilon) $.} when $ \cU_\varepsilon^0=\PP \, \cU_0^0 + \PP \, o(1) + \mathcal{O}(\varepsilon)$ in $ H^s$, we obtain that
\[
\sup_{\eps \in ]0,\eps_0]} \ \bigl( \parallel \cU_\eps \parallel_{L^\infty([0,T]; H^s)} + \parallel \cU_\eps \parallel_{W^{1,\infty}([0,T]; H^{s-1})} \bigr) < + \infty \,.
\]
Indeed, these energy estimates rely on the symmetry of the system \eqref{TFEMS3} and the fact that the singular linear operator $\varepsilon^{-1} \cL$ does not contribute to energy estimates since it is skew-adjoint and has constant coefficients. This uniform control allows to prove two important results (see Appendix A).

First, as soon as the family $ \{\cU_\varepsilon^0\}_\eps $ is bounded in $ H^s$,  there exist unique classical  solutions $\cU_\varepsilon \in \mathscr{C}([0,T]; H^s)$ to \eqref{TFEMS3} with prepared initial data $\cU_\varepsilon^0 $. Second, using some compactness arguments (namely compact Sobolev embeddings, Ascoli--Arzela theorem and interpolation inequalities),  as $\varepsilon \rightarrow 0$, the solution $\cU_\varepsilon$ converges in $\mathscr{C}([0,T]; H^{s-\delta})$ for every $\delta>0$ to some profile $\cU_0\in \mathscr{C}([0,T]; H^{s})$.
By passing to the limit $\varepsilon \rightarrow 0$ in \eqref{TFEMS3} via the strong compactness of $\cU_\varepsilon$ in $\mathscr{C}([0,T]; H^s)$, this expression $ \cU_0 $ turns out to be the unique solution of
\begin{equation}
  \label{modu0prepa2}
  \qquad  \partial_t (\PP \, \cU_0) =   \sum_{j=1}^3 \PP \cA_j(0,\PP\,\cU_0) \,
  \part_{x_j} \PP\,\cU_0 + \PP \cF(0,\PP\,\cU_0)\, , \qquad \PP \, {\cU_0}_{\mid t= 0} = \PP \, \cU^0_0 \,.
\end{equation}
This formulation of SLM is not really optimal\footnote{As will be seen, some equations inside SLM are redundant or inactive in our context.}. Indeed, it does not take into account the influence of the relation \eqref{GaussLawSS}. As a matter of fact, an asymptotic  consequence of \eqref{defL} is the condition $ \cG \,\cU_0 = 0 $. The profile $ \cU_0 $ must be polarized inside $\cH := {\rm Ker}\, \cL\cap {\rm Ker}\,  \cG $.

In other words, we must impose $\cU_0\in \cH $. From now on, we denote by $ \PP_{\! e} $ (with the index ``$e$'' for \underline{e}ffective) the orthogonal projector onto the subspace $ \cH $. By construction, we already know that $ \cU_0^0 = \PP_{\! e} \, \cU_0^0 $, and we want to guarantee that $ \cU_0 = \PP_{\! e} \, \cU_0 $. Thus, instead of looking directly at SLM, we prefer to consider the Effective Slow Limit Model (ESLM) given by
\begin{equation}
  \label{modu0effective}
 \qquad  \partial_t (\PP_{\! e} \, \cU_0) =   \sum_{j=1}^3 \PP_{\! e} \, \cA_j(0,\PP_{\! e} \, \cU_0) \, \part_{x_j} \PP_{\! e} \, \cU_0 + \PP_{\! e} \, \cF(0,\PP_{\! e} \,\cU_0)\, , \quad \ \PP_{\! e} \, {\cU_0}_{\mid t= 0} = \PP_{\! e} \, \cU^0_0 \,.
\end{equation}
This is again a quasilinear symmetric system and, as such, it is locally well-posed in the space $ H^s $ as soon as $ s > 5/2 $. It is more convenient to implement ESLM than SLM not only because ESLM produces systematically a polarization onto the required subspace $ \cH $, but especially because the reduced image of $ \PP_{\! e} $ (in comparison with the image of $ \PP $) gives access to additional simplifications that are helpful in the computations.
For these reasons, in practice, we will first  compare ESLM with incompressible XMHD. It is only after that we will check that the solution of \eqref{modu0effective} coincides with the one issued from  \eqref{modu0prepa2}.


\subsection{The orthogonal projectors}
\label{subs:PP}
The aim here is to derive  (relatively simple) formulas for $\PP_{\! e} $ and $\PP $. Consider a periodic function
\[
\cU_0 \in L^2 (\TT^3;\RR^{14}) \, , \qquad \cU_0 (x) = {}^t(\varrho_{e0},\varrho_{i0},{}^t u_{e0},{}^t u_{i0},{}^t E_0 ,{}^t B_{0}) (x) \,,
\]
with Fourier decomposition
\begin{equation*}
\label{Fourierdecomp} \qquad \ \, \cU_0(x) = \sum_{k \in \ZZ^3} \rwidehat{\cU}_{0k} \ e^{i k \cdot x} \, , \qquad \rwidehat{\cU}_{0k} = {}^t(\hat\varrho_{e0k},\hat\varrho_{i0k},{}^t \widehat u_{e0k},{}^t \widehat u_{i0k},{}^t \widehat E_{0k} ,{}^t\widehat B_{0k}) \in \CC^{14} \,.
\end{equation*}
Since $ \cU_0 $ is real valued, we know already that
\[ \rwidehat{\cU}_{0k} = \overline{\rwidehat{\cU}}_{0(-k)} \, , \qquad \forall k \in \ZZ \, .
\]
In particular, we can isolate the mean value of $ \cU_0 $ which is
\begin{equation*}
\label{defmeanop}
\ \ \overline{\cU}_0 := \rwidehat{\cU}_{00} = \frac{1}{(2 \pi)^3} \int_{\TT^3} \cU_0(x) \, dx = {}^t(\bar{\varrho}_{e0},\bar{\varrho}_{i0},{}^t\overline{u}_{e0},{}^t\overline{u}_{i0},{}^t\myoverbar{E}_0 ,{}^t\myoverbar{B}_{0}) \in \RR^{14} \, .
\end{equation*}
In Lemma \ref{lem:CPPe} below and in the rest of the text, we will denote by $ \LL $ the \underline{L}eray projection   onto divergence-free vector fields, that is
\[ \LL \, u := u - \nabla \Delta^{-1} (\nabla \cdot u) \, . \]
Recall also that the coefficients $ \underline{b} $ and $\underline{c} $ have been defined in \eqref{defabbar}.

\begin{lem}[Computation of $\PP_{\! e} $]
  \label{lem:CPPe}
  For all $ \cU_0 \in L^2 (\TT^3;\RR^{14}) $, we find
\begin{equation}
\label{defPU0:2}
\qquad \quad \PP_{\! e}\, \cU_0 =\left(\begin{array}{c}
 \underline{c}\, \sqrt{{\underline{p}_e'}/{\underline{p}_i'}} \ \bar{\varrho}_{e0}+ \underline{c} \ \bar{\varrho}_{i0} \medskip \\
 \underline{c} \ \bar{\varrho}_{e0}+ \underline{c} \,  \sqrt{{\underline{p}_i'}/{\underline{p}_e'}} \  \bar{\varrho}_{i0}\medskip\\
  \displaystyle \frac{\sqrt{m_e}}{m_e+m_i} \, \LL \big(\sqrt{m_e} \, u_{e0} +\sqrt{m_i} \, u_{i0} \big)
  - \frac{\underline{b}}{1-\underline{b} \, \Delta} \,
  \sqrt{\frac{\underline{n}}{m_e}} \, \nabla \times B_0 \smallskip \\
  \displaystyle  \qquad \qquad - \,  \frac{\underline{b}^2}{1-\underline{b} \, \Delta} \ \frac{\underline{n}}{m_e} \ \Delta \LL u_{e0} + \frac{\underline{b}^2}{1-\underline{b} \, \Delta} \  \frac{\underline{n}}{\sqrt{m_e \, m_i}} \ \Delta \LL u_{i0}
  \medskip \\
  \displaystyle \frac{\sqrt{m_i}}{m_e+m_i} \, \LL \big(\sqrt{m_e} \,  u_{e0} +\sqrt{m_i} \, u_{i0} \big)
  + \frac{\underline{b}}{1-\underline{b} \, \Delta} \, \sqrt{\frac{\underline{n}}{m_i}} \, \nabla \times B_0 \smallskip \\
  \displaystyle  \qquad \qquad + \,  \frac{\underline{b}^2}{1-\underline{b} \, \Delta} \
   \frac{\underline{n}}{\sqrt{m_e \, m_i}}
   \ \Delta \LL u_{e0} - \frac{\underline{b}^2}{1-\underline{b} \, \Delta} \ \frac{\underline{n}}{m_i} \ \Delta \LL u_{i0} \medskip \\
  0 \medskip \\
   \displaystyle \frac{1}{1-\underline{b}\, \Delta} \, \bigg( \LL B_0 + \underline{b} \, \nabla \times
  \bigg( \sqrt{\frac{\underline{n}}{m_i}} \, u_{i0}-\sqrt{\frac{\underline{n}}{m_e}}\, u_{e0}\bigg)
  \bigg)
\end{array}\right)\,,
\end{equation}
where the third and fourth lines (as well as the fifth and sixth lines) in this column vector must be put together along the same line (they are separated for presentational reasons).
\end{lem}

The rest of Subsection \ref{subs:PP}, including Paragraph \ref{Para:0} up to  \ref{Para:3} (without  Corollary~\ref{furthersimplification}), is devoted to the proof of Lemma~\ref{lem:CPPe}. Since both $\cL$ and $ \cG $ are linear operators with real  constant coefficients acting on periodic real valued functions, we can decompose any linear map $ {\mathcal M } \in \{ \PP_{\! e}, \cL, \cG \} $ mode by mode on Fourier series according to the following definition (of the complex valued matrices $ {\mathcal M}_k $)
\[ 
{\mathcal M} \,\cU_0 (x) = \sum_{k \in \ZZ^3} {\mathcal M}_k \, \rwidehat{\cU}_{0k} \ e^{i k \cdot x} \, , \qquad {\mathcal M}_k \, \rwidehat{\cU}_{0k} = \myoverbarM{\mathcal M}_{-k} \, \overline{\rwidehat{\cU}}_{0(-k)}\,.
\]
From the construction of $ \cL $, we obtain
\begin{equation*}
  {\cL}_k = \left( \begin{array}{cccccc}
   0 & 0 & - {\rm i} \,  \underline{a}_e \, {}^tk & 0 & 0 & 0 \\
   0 &  0 & 0 &   - {\rm i} \, \underline{a}_i \, {}^tk & 0 & 0 \\
     - {\rm i}\, \underline{a}_e \, k & 0 & 0 & 0 &  \displaystyle-\sqrt{\frac{\underline{n}}{m_e}}\,{\rm I}_3 & 0 \\
   0 &   - {\rm i} \,  \underline{a}_i \, k & 0 & 0  &  \displaystyle\ \ \, \sqrt{\frac{\underline{n}}{m_i}}\,{\rm I}_3 & 0 \\
    0 & 0 & \displaystyle \sqrt{\frac{\underline{n}}{m_e}}\,{\rm I}_3 &  \displaystyle- \sqrt{\frac{\underline{n}}{m_i}}\,{\rm I}_3 & 0  &  {\rm i}\,k\times\\
    0 & 0 & 0 & 0 &  -{\rm i}\, k\times & 0 
  \end{array}  \right) .
\end{equation*}
From \eqref{defL}, we also obtain
\begin{equation*}
  {\cG}_{k}:=
  \left( \begin{array}{cccccc}
    0 \, & 0 & 0 & 0 & 0 &  {\rm i} \,{}^tk  \\
    \displaystyle \sqrt{\underline{n}/\underline{p}'_e} \, &  \displaystyle -\sqrt{\underline{n}/\underline{p}'_i}  & 0 & 0 &  {\rm i}\, {}^tk  & 0
  \end{array}  \right).
\end{equation*}
The access to $ \PP_{\! ek} $ is driven by the determination of the subspace $ {\cH}_{k} := {\rm Ker}\, {\cL}_k \cap {\rm Ker}\, {\cG}_k $. Below, for all $ k \in \ZZ^3 $, we exhibit an  orthonormal basis of $ {\cH}_{k} $. We first consider the case of $\cH_0$, and then separately the $\cH_k$ with  $k\neq 0$.


\subsubsection{Subspace ${\cH}_{0}$}\label{Para:0}
Consider a vector
\[
\rwidehat{\cU}_{00} = \overline{\rwidehat{\cU}}_{00} = \overline{\cU}_0 ={}^t(\bar{\varrho}_{e0},\bar{\varrho}_{i0},{}^t \overline{u}_{e0},{}^t \overline{u}_{i0},{}^t \myoverbar{E}_0,{}^t \myoverbar{B}_0) \in \RR^{14} \, ,
\]
which is a solution of the two equations $ {\cL}_{0} \, \overline{\cU}_0 = 0 $ and $ {\cG}_{0} \, \overline{\cU}_0 =0$. We find $ \myoverbar{E}_0 = 0 $. Moreover, there must be three constants $\bar{\varrho} \in \RR $, $\overline{u} \in\RR^3$  and   $\myoverbar{B} \in\RR^3$ such that
\begin{equation}
  \label{threeconstants}
\quad \bar{\varrho}_{e0}=\sqrt{{\underline{p}_e'}/{\underline{p}_i'}}\, \bar{\varrho} \,, \quad \
\bar{\varrho}_{i0}=\bar{\varrho}\,,\quad \
\overline{u}_{e0} = \sqrt{m_e/m_i}\, \overline{u}\,, \quad \   \overline{u}_{i0}=\overline{u}\,, \quad \ \myoverbar{B}_0 = \myoverbar{B}\,.
\end{equation}
Therefore, the subspace  $ {\cH}_{0} ={\rm vect} \, \{ w_0^j\}_{j\in\{0,\ldots,6\}} $ is of dimension $ 7 $ (on $ \RR $), generated by
\[
\begin{array}{ll}
  \displaystyle {}^t w_0^0 = \big(1+{\underline{p}_e'}/{\underline{p}_i'}\big)^{-\frac{1}{2}} \ {}^t\Big(\sqrt{{\underline{p}_e'}/{\underline{p}_i'}},1,0,0,0,0\Big) \,, & \smallskip \\
  \displaystyle  {}^tw_0^j = \frac{\sqrt{m_i}}{\sqrt{m_e+m_i}}\, {}^t\Big(0,0,\sqrt{m_e/m_i}\, {}^t e^j,{}^t e^j,0,0\Big)\,, \quad & \forall j \in \{1,2,3\} \, , \smallskip \\
                 {}^tw_0^j = {}^t(0,0,0,0,{}^t e^{j-3})\,, & \forall j \in \{4,5,6\} \, .
\end{array}
\]
From there, it is easy to compute $\PP_{\! e 0} :\RR^{14} \rightarrow \RR^{14}$, which is
\begin{equation*}
  \PP_{\! e 0}=\sum_{j=0}^6 w_0^j \,{}^t w_0^j =
  \left( \begin{array}{cccccc}
    \underline{c}\sqrt{{\underline{p}_e'}/{\underline{p}_i'}} & \underline{c} & 0 & 0 & 0 & 0 \\
    \underline{c} &   \underline{c}\sqrt{{\underline{p}_i'}/{\underline{p}_e'}} & 0 & 0 & 0 & 0 \\
    0 & 0 & \frac{m_e}{m_e+m_i}\, {\rm I}_3  &  \frac{\sqrt{m_e m_i}}{m_e+m_i}\, {\rm I}_3  & 0 & 0 \medskip\\
    0 & 0 &  \frac{\sqrt{m_e m_i}}{m_e+m_i}\, {\rm I}_3  &   \frac{ m_i}{m_e+m_i}\, {\rm I}_3  & 0& 0 \\
    0 & 0 & 0 & 0 & 0  &  0\\
    0 & 0 & 0 & 0 & 0 & {\rm I}_3   
  \end{array}  \right) \,.
\end{equation*}  
Thus, we can assert that
\begin{equation}
  \label{eqn:ctepart}
  \PP_{\! e 0}\, \overline{\cU}_0 =\left(\begin{array}{c}
    \underline{c}\sqrt{{\underline{p}_e'}/{\underline{p}_i'}} \ \bar{\varrho}_{e0}+ \underline{c} \ \bar{\varrho}_{i0} \\
    \underline{c} \ \bar{\varrho}_{e0}+ \underline{c} \sqrt{{\underline{p}_i'}/{\underline{p}_e'}} \  \bar{\varrho}_{i0}\medskip\\
    \displaystyle \frac{\sqrt{m_e}}{m_e+m_i} \big(\sqrt{m_e} \, \overline{u}_{e0} +\sqrt{m_i} \, \overline{u}_{i0} \big)\medskip \\
    \displaystyle \frac{\sqrt{m_i}}{m_e+m_i} \big(\sqrt{m_e} \, \overline{u}_{e0} +\sqrt{m_i} \, \overline{u}_{i0} \big) \\
    0\\
    \displaystyle \myoverbar{B}_0
  \end{array}\right) \,.
\end{equation}
To have access to $ \PP_0 $, it suffices to drop the condition on $ (\bar{\varrho}_{e0},\bar{\varrho}_{i0}) $ induced by $ \cG_0 $. Therefore
\begin{equation}
  \label{eqn:ctepartdrop}
  \PP_{0}\, \overline{\cU}_0 =\left(\begin{array}{c}
    \bar{\varrho}_{e0} \\
  \bar{\varrho}_{i0}\medskip\\
  \displaystyle \frac{\sqrt{m_e}}{m_e+m_i} \big(\sqrt{m_e} \, \overline{u}_{e0} +\sqrt{m_i} \, \overline{u}_{i0} \big)\medskip \\
  \displaystyle \frac{\sqrt{m_i}}{m_e+m_i} \big(\sqrt{m_e} \, \overline{u}_{e0} +\sqrt{m_i} \, \overline{u}_{i0} \big) \\
  0\\
  \displaystyle \myoverbar{B}_0
  \end{array}\right)\,.
\end{equation}


\subsubsection{Subspaces ${\cH}_{k}$ with $k\neq 0$}
\label{Para:1}
Let $ \rwidehat{\cU}_{0k} = \cU={}^t(\varrho_e,\varrho_i,{}^tu_e,{}^tu_i,{}^tE,{}^tB)$ be a solution\footnote{For convenience, we drop in Paragraph \ref{Para:1} the hat  ``$ \widehat{\ \, \ } $'' and the subscript ``$ 0k $'' in $ \widehat{\cU}_{0k} $.} of both equations $ {\cL}_{k} \,\cU =0$ and $ {\cG}_{k} \,\cU =0$. Concretely, the vector $\cU \in \CC^{14} $ must satisfy  the following set of equations
\[\left\{
\begin{array}{ll}
  \displaystyle {\rm i}\, k\cdot u_e =0\,, \quad &  {\rm i}\, k\cdot u_i =0\,, \smallskip \\
  \displaystyle {\rm i}\, k \,  \sqrt{\frac{\underline{p}_e'}{\underline{n}}} \,\varrho_e + E=0\,, \quad & \displaystyle {\rm i}\, k \,  \sqrt{\frac{\underline{p}_i'}{\underline{n}}} \ \varrho_i - E=0\,, \\ \displaystyle {\rm i}\, k\cdot E=  \sqrt{\frac{\underline{n}}{\underline{p}_i'}} \ \varrho_i -\sqrt{\frac{\underline{n}}{\underline{p}_e'}} \ \varrho_e\,, \quad & {\rm i}\, k\times E =0\,, \\ \displaystyle {\rm i}\, k\cdot B =0\,, \quad & \displaystyle {\rm i}\, k\times B +  \sqrt{\frac{\underline{n}}{m_e}}\,u_e -  \sqrt{\frac{\underline{n}}{m_i}}\,u_i = 0 \,.
\end{array}\right.
\]
From the second line, we can extract
\[
  {\rm i}\, \vert k \vert^2 \,  \sqrt{\frac{\underline{p}_e'}{\underline{n}}} \,\varrho_e + k \cdot E =0\,, \qquad  {\rm i}\, \vert k \vert^2 \,  \sqrt{\frac{\underline{p}_i'}{\underline{n}}} \ \varrho_i - k \cdot E =0\,, \qquad {\rm i}\, k\cdot E=  \sqrt{\frac{\underline{n}}{\underline{p}_i'}} \ \varrho_i -\sqrt{\frac{\underline{n}}{\underline{p}_e'}} \ \varrho_e\, .
\]
There is only a trivial  solution for this system of three equations with three unknowns, given by $ \varrho_e =0 $, $ \varrho_i =0 $ and $ k \cdot E =0 $. Since $ k\times E =0 $, we must have $ E = 0 $.

To manage the remaining conditions, we  introduce a right handed orthonormal  basis\footnote{This basis  should not be confused with the canonical basis $ (e^1,e^2,e^3) $ of $ \RR^3 $.} $ \bigl( \mathbbm{e}^1(k),\mathbbm{e}^2(k), \mathbbm{e}^3(k) \bigr) $  adjusted in such a way that $\mathbbm{e}^1(k):=k/|k|$ and $ \mathbbm{e}^1(k) \times \mathbbm{e}^2(k) = \mathbbm{e}^3(k) $.  The previous system leads to
\[
\quad B=\frac{{\rm i}}{|k|}\,\mathbbm{e}^1\times \bigg(\sqrt{\frac{\underline{n}}{m_i}}\,u_i -  \sqrt{\frac{\underline{n}}{m_e}}\,u_e\bigg)\,, \qquad
u_s = \alpha_s(k) \,\mathbbm{e}^2(k) +\beta_s(k) \, \mathbbm{e}^3(k)\,, \quad  s\in\{e,i \} \, ,
\]
where the scalar coefficients $\alpha_s$ and $\beta_s$ with $ s\in\{e,i \} $ are four  arbitrary complex numbers. The subspace ${\cH}_{k} ={\rm vect}\, \{ \widetilde{W}_k^j\}_{j\in\{1,\ldots,4\}}$ is therefore of dimension $ 4 $ (on $ \CC $), generated by
\[
\begin{array}{ll}
  {}^t\widetilde{W}_k^1 := {}^t\big(0,0,{}^t\mathbbm{e}^2,0,0,-{\rm i}|k|^{-1}\sqrt{\underline{n}/m_e}\,{}^t\mathbbm{e}^3\big)\,, \quad &
  {}^t\widetilde{W}_k^2 :=  {}^t\big(0,0,{}^t\mathbbm{e}^3,0,0,\ \, {\rm i}|k|^{-1}\sqrt{\underline{n}/m_e}\, {}^t\mathbbm{e}^2\big)\,, \medskip \\
  {}^t\widetilde{W}_k^3 := {}^t\big(0,0,0,{}^t\mathbbm{e}^2,0, \ \, {\rm i}|k|^{-1}\sqrt{\underline{n}/m_i}\, {}^t\mathbbm{e}^3\big)\,, &
  {}^t\widetilde{W}_k^4 :=  {}^t\big(0,0,0, {}^t\mathbbm{e}^3,0,-{\rm i}|k|^{-1}\sqrt{\underline{n}/m_i}\, {}^t\mathbbm{e}^2\big)\,.
\end{array}
\]
In order to obtain an orthogonal basis of ${\cH}_{k} $ and also to partially separate velocities from the magnetic field\footnote{The vectors $ {W}_k^1 $ and $ {W}_k^2 $ are polarized onto the two-fluid velocities.}, we prefer to choose the following eigenvectors of $ \cL_k $, obtained by combining adequately the $ \widetilde{W}_k^j $, namely
\begin{eqnarray*}
  &{W}_k^1 :=\sqrt{m_e}\,\widetilde{W}_k^1+\sqrt{m_i}\,\widetilde{W}_k^3\,, \quad &{W}_k^2 :=\sqrt{m_e}\,\widetilde{W}_k^2+\sqrt{m_i}\,\widetilde{W}_k^4\,, \smallskip \\
 &\qquad \  {W}_k^3 :=\sqrt{\underline{n}/m_e}\,\widetilde{W}_k^1-\sqrt{\underline{n}/m_i}\,\widetilde{W}_k^3\,, \quad &{W}_k^4 :=\sqrt{\underline{n}/m_e}\,\widetilde{W}_k^2-\sqrt{\underline{n}/m_i}\,\widetilde{W}_k^4\,,
\end{eqnarray*}  
which, once normalized, give rise to
\begin{eqnarray*}
{}^tw_k^1&:=&    \big(m_e+m_i \big)^{-\frac{1}{2}} \, {}^t\Big(0,0,\sqrt{m_e}\,{}^t\mathbbm{e}^2,\sqrt{m_i}\,{}^t\mathbbm{e}^2,0,0\Big)\,,\\
  {}^tw_k^2&:=& \big(m_e+m_i \big)^{-\frac{1}{2}} \,{}^t\Big(0,0,\sqrt{m_e}\,{}^t\mathbbm{e}^3,\sqrt{m_i}\,{}^t\mathbbm{e}^3,0,0\Big)\,,\\
  {}^tw_k^3&:=& \Big(\underline{b}^{-1}+\underline{b}^{-2}|k|^{-2} \Big)^{-\frac{1}{2}} \,{}^t\Big(0,0,\sqrt{\underline{n}/m_e}\,{}^t\mathbbm{e}^2,-\sqrt{\underline{n}/m_i}\,{}^t\mathbbm{e}^2, 0,- {\rm i}(\underline{b}|k|)^{-1}\,{}^t\mathbbm{e}^3\Big)\,,\\
  {}^tw_k^4&:=& \Big(\underline{b}^{-1}+\underline{b}^{-2}|k|^{-2} \Big)^{-\frac{1}{2}} \,{}^t\Big(0,0,\sqrt{\underline{n}/m_e}\,{}^t\mathbbm{e}^3,-\sqrt{\underline{n}/m_i}\,{}^t\mathbbm{e}^3, 0,\ \ {\rm i}(\underline{b}|k|)^{-1}\,{}^t\mathbbm{e}^2\Big)\,.
\end{eqnarray*}
Observe that the two eigenvectors $ w_k^3 $ and $ w_k^4 $ are complex valued. We denote by $ \overline{w}_k^j $ the complex conjugate of $ w_k^j $. The orthogonal projector $\PP_{\! ek}:\CC^{14} \rightarrow \CC^{14}$ onto the subspace ${\cH}_{k}$ is then given by
\begin{equation}\label{orthoprojek}
\PP_{\! e k}=\sum_{j=1}^4 w_k^j \,{}^t \overline{w}_k^j\,.
\end{equation}


\subsubsection{Subspaces $\text{\rm Ker} \, {\cL}_{k}$ with $k\neq 0$}
\label{Para:2}
Of course, we have $ \cH_k \subset \text{\rm Ker} \, {\cL}_{k} $. The subspace $ \text{\rm Ker} \, {\cL}_{k} $ is accessed from $ \cH_k $ by removing the two linearly independent conditions issued from $ \cG_k $. Thus, to recover $ \text{\rm Ker} \, {\cL}_{k} $, it suffices to complete the basis of $ \cH_k $ with two well-chosen vectors. We can select
\[
\begin{array}{ll}
  {}^tw_k^5 := {}^t (0,0,0,0,0,{}^t\mathbbm{e}^1) \, , \smallskip \\
  \displaystyle {}^tw_k^6 := \big(\underline{n} \,  \underline{p}'_e + \underline{n} \,  \underline{p}'_i + \vert k \vert^2 \,  \underline{p}'_e \, \underline{p}'_i \bigr)^{-\frac{1}{2}} \ {}^t \Big(-\sqrt{\underline{n} \, \underline{p}'_i} \, , \, \sqrt{\underline{n} \, \underline{p}'_e} \, , \, 0 \, , \, 0 \, , \, i \, \vert k \vert \, \sqrt{\underline{p}'_e \, \underline{p}'_i} \ {}^t\mathbbm{e}^1 \, ,\, 0 \, \Big) \, .
\end{array}
\]
It is clear that the six  vectors $ w_k^j $ with $ j \in\{1,\cdots,6 \} $ generate an orthonormal basis of $ \text{\rm Ker} \, {\cL}_{k} $. It follows that
\begin{equation}
  \label{PPversusPPe}
  \PP_{k}= \PP_{\! e k}+  w_k^5 \,{}^t \overline{w}_k^5 + w_k^6 \,{}^t \overline{w}_k^6 \,.
\end{equation}
Define
\begin{equation}
  \label{cestpourcK}
  \cK := \bigl \{ \, \cU_0 \in L^2(\TT^3;\RR^{14}) \, ; \, \cU_0 = {}^t (0,0,u_{e0},u_{i0},E_0,0) \ \text{and} \ \nabla \cdot E_0 = 0 \, \} \, .
\end{equation}
\begin{lem}[The actions of the projectors $ \PP $ and $ \PP_{\! e} $ coincide on $ \cK $]
  \label{actions}
  Assume that $ \cU_0 \in \cK $ with $ \cK $ as in \eqref{cestpourcK}. Then, we have $ \PP \, \cU_0 = \PP_{\! e} \, \cU_0 $.
\end{lem}
\begin{proof} Select $ \cU_0 \in \cK $. In particular, we find $ \bar{\varrho}_{e0} = \bar{\varrho}_{i0} =0 $.
From there, comparing \eqref{eqn:ctepart} and \eqref{eqn:ctepartdrop}, we see that $ \PP_0 \, \overline{\cU}_0 = \PP_{\! e 0} \, \overline{\cU}_0 $. Secondly, for $ k \not = 0 $, remark that $ {}^t \overline{w}_k^5 \, \rwidehat{\cU}_{0k} = 0 $, simply because $ \widehat{B}_{0k} = 0 $. Moreover, we can assert that $ {}^t \overline{w}_k^6 \, \rwidehat{\cU}_{0k} = 0 $ because $ \widehat{\varrho}_{i0k} = \widehat{\varrho}_{e0k} = 0 $ and $ \vert k \vert \, \mathbbm{e}^1 \cdot \widehat{E}_{0k} = k \cdot \widehat{E}_{0k} = 0 $ (since $ \nabla \cdot E_0 = 0 $). Thus, in view of \eqref{PPversusPPe}, we are sure that $ \PP_{k} \, \rwidehat{\cU}_{0k} = \PP_{\! e k} \, \rwidehat{\cU}_{0k} $.
\end{proof}


\subsubsection{Reconstitution of $ \PP_{\! e} $}
\label{Para:3}
By construction, we find
\[
\PP_{\! e} \, \cU_0 = \PP_{\! e 0} \, \rwidehat{\cU}_{00} + \PP^v_{\! e} \, \cU_0 + \PP^m_{\! e} \, \cU_0 \, ,
\]
where the projectors $ \PP^v_{\! e} $ and $ \PP^m_{\! e} $ serve to distinguish respectively \underline{v}elocity and (remaining) \underline{m}agnetic parts inside $ \PP_{\! e} \, \cU_0 - \PP_{\! e 0} \, \rwidehat{\cU}_{00} $. More precisely (with $\ZZ_*^3\equiv \ZZ^3\setminus  \{ 0 \}$),
\[
\PP_{\! e}^v \, \cU_0 := \sum_{k\in \ZZ_*^3} e^{{\rm i}k\cdot x} \, \sum_{j=1,2} ({}^t w_k^j \, \rwidehat{\cU}_{0k}) \ w_k^j  \, , \qquad
\PP_{\! e}^m \, \cU_0 := \sum_{k\in \ZZ_*^3} e^{{\rm i}k\cdot x} \, \sum_{j=3,4} ({}^t \overline{w}_k^j \,  \rwidehat{\cU}_{0k}) \ w_k^j \,.
\]
A straightforward calculation indicates that
\begin{equation}
  \label{eqn:T1}
  \PP_{\! e}^v \, \cU_0 = \sum_{k\in \ZZ_*^3} e^{{\rm i}k\cdot x} \ \frac{1}{m_e+m_i}
  \left( \begin{array}{l}
    0 \\ 0 \\ \LL_k \, \big(m_e \, \widehat{u}_{e0k}+\sqrt{m_e m_i}\,\widehat{u}_{i0k}\big) \smallskip\\
    \LL_k \,  \big(\sqrt{m_e m_i}\,\widehat{u}_{e0k}+m_i \, \widehat{u}_{i0k}\big)\\ 0 \\ 0
  \end{array}  \right) \,, \end{equation}
\vskip -5mm
where
\[
\LL_k:=\sum_{j=2,3}\mathbbm{e}^j(k)\otimes\mathbbm{e}^j(k)={\rm Id}-\mathbbm{e}^1(k)\otimes\mathbbm{e}^1(k) \, , \qquad \forall k \in \ZZ_\ast^3 \, .
\]
Remark that, with the convention $ \LL_0 := {\rm Id} $, the Fourier multiplier built with the action of $ \LL_k $ on the $k$-th Fourier coefficient corresponds to the Leray projector $ \LL $ onto divergence-free vector fields. Observe also that Amp\`ere's law
\begin{equation}
  \label{eqn:ampere}
  {\rm i}\, k\times \widehat{B}_{0k} = \widehat{J}_{0k}\,, \qquad
 \widehat{J}_{0k}:= \sqrt{\frac{\underline{n}}{m_i}}\,\widehat{u}_{i0k} -\sqrt{\frac{\underline{n}}{m_e}}\,\widehat{u}_{e0k}\, ,
\end{equation}
which enters in the definition of the space $\cH_k$ is, as expected, satisfied by $ \PP_{\! e k}^v \, \rwidehat{\cU}_{0k} $. This is an opportunity to introduce the field $ J_0 $ which is built with the Fourier coefficients $\widehat{J}_{0k} $ and which, in the context of \eqref{TFEMS3}, can be viewed as a current density.

We continue with $ \PP_{\! e}^m $. Using the definition of $w_k^3$ and  $w_k^4$, we obtain
\[
\begin{array}{l}
{}^t \overline{w}_k^3 \, \rwidehat{\cU}_{0k} =
\displaystyle \ \ \frac{\underline{b}|k|}{\sqrt{(1+\underline{b}|k|^2)}} \  \Big\{{\rm i}(\underline{b}|k|)^{-1} \widehat{B}_{0k} \cdot \mathbbm{e}^3(k) - \widehat{J}_{0k} \cdot \mathbbm{e}^2(k)\Big\}\,, \smallskip \\
              {}^t \overline{w}_k^4 \, \rwidehat{\cU}_{0k} = \displaystyle-\frac{\underline{b}|k|}{\sqrt{(1+\underline{b}|k|^2)}} \ \Big\{{\rm i}(\underline{b}|k|)^{-1} \widehat{B}_{0k} \cdot \mathbbm{e}^2(k)  + \widehat{J}_{0k} \cdot \mathbbm{e}^3(k)\Big\}\, .
\end{array}
\]
It follows that
\begin{equation}
  \label{eqn:T2}
  \PP_{\! e}^m \, \cU_0  = \sum_{k\in \ZZ_*^3} \frac{e^{{\rm i}k\cdot x}}{1+\underline{b}|k|^2}
  \left( \begin{array}{c}
    0 \\ 0 \\  \displaystyle - \sqrt{\underline{n}/m_e} \ \Bigl( {\rm i} \,   \underline{b} \ k \times \widehat{B}_{0k} + \underline{b}^2 \, \vert k \vert^2 \ \LL_k \widehat{J}_{0k} \Bigr)  \medskip\\
    \displaystyle \ \ \sqrt{\underline{n}/m_i} \ \Bigl( {\rm i} \,   \underline{b} \ k \times \widehat{B}_{0k} + \underline{b}^2 \, \vert k \vert^2 \ \LL_k \widehat{J}_{0k} \Bigr) \\ 0 \\
    \displaystyle  \LL_k \widehat{B}_{0k} + {\rm i}\underline{b}\, k\times \widehat{J}_{0k}\\
  \end{array} \right)\,.
\end{equation}
Gathering \eqref{eqn:ctepart}, \eqref{eqn:T1} and \eqref{eqn:T2}, we obtain \eqref{defPU0:2}. It is straightforward to verify as expected that $\PP_{\! e}=\PP^\ast_{\! e} $, and that $ \PP_{\! e} \, \cU_0 $ is indeed in the kernel of $ \cL $, i.e., $\PP \PP_{\! e} = \PP_{\! e}$. Since $ \PP_{\! e} $ is constructed from \eqref{orthoprojek}, it is sure\footnote{It is somewhat  cumbersome to verify directly that $\PP_{\! e}^2=\PP_{\! e}$.} that $\PP_{\! e}^2=\PP_{\! e}$. This completes the proof of Lemma~\ref{lem:CPPe}.

The formula \eqref{defPU0:2} holds for all $ \cU_0 \in L^2 (\TT^3;\RR^{14} ) $. But in practice, when dealing with \eqref{modu0effective}, we manipulate expressions $ \PP_{\! e} \, \cU_0 $ which are adjusted in such a way that $ \cU_0 = \PP_{\! e} \, \cU_0 $. It is  advantageous to exploit this {a priori} information.

\begin{cor}[Further simplification of $ \PP_{\! e} $]
  \label{furthersimplification}
  Assume that $ \cU_0 = \PP_{\! e} \, \cU_0 $. Then, we have
  \begin{equation}
    \label{defPU0:2simpli}
    \PP_{\! e} \, \cU_0 =\left(\begin{array}{c}
      \bar{\varrho}_{e0} \medskip \\ \bar{\varrho}_{i0}\medskip\\
      \displaystyle \frac{\sqrt{m_e}}{m_e+m_i} \LL \big(\sqrt{m_e} u_{e0} +\sqrt{m_i} u_{i0} \big)
      -   \sqrt{\frac{\underline{n}}{m_e}} \, \underline{b} \,\nabla \times B_0
      \medskip \\
      \displaystyle \frac{\sqrt{m_i}}{m_e+m_i} \LL \big(\sqrt{m_e} u_{e0} +\sqrt{m_i} u_{i0} \big)
      + \sqrt{\frac{\underline{n}}{m_i}} \, \underline{b} \, \nabla \times B_0  \medskip \\
      0 \medskip \\
      \displaystyle \LL B_0
    \end{array}\right) \,.
  \end{equation}
\end{cor}

\begin{proof} Since $ \cU_0 = \PP_{\! e} \, \cU_0 $ is polarized inside $\cH$, we can exploit  \eqref{threeconstants}. This furnishes the two first components of \eqref{defPU0:2simpli}. Since $ \cU_0 = \PP_{\! e} \, \cU_0 $, we can also implement Amp\`ere's law \eqref{eqn:ampere} which in our context reads
  \begin{equation}
    \label{Ampereslaw}
    \nabla \times B_0 = J_{0}\,, \qquad
    J_0 := \sqrt{\frac{\underline{n}}{m_i}}\, u_{i0} -\sqrt{\frac{\underline{n}}{m_e}}\, u_{e0} \, .
  \end{equation}
  Just plug \eqref{Ampereslaw} inside \eqref{defPU0:2} to see cancellations which lead to \eqref{defPU0:2simpli}.
\end{proof}


\subsection{The modulation equations}
\label{sec:modulation}
In Paragraph \ref{subs:contentESLM}, we make explicit the coefficients involved in the effective slow limit  equations. In Paragraph \ref{subs:end}, starting from prepared initial data $ \cU_0^0 = \PP_{\! e} \, \cU_0^0 $, we show that ESLM, SLM and XMHD are equivalent systems. In Paragraph \ref{subs:irroflows}, we consider separately what happens in the case of irrotational flows.


\subsubsection{Computation of the ESL Model}
\label{subs:contentESLM}
We are now ready to detail the content  of \eqref{modu0effective}. This can be done by inserting $ \cU_0 = \PP_{\! e} \, \cU_0 $ with $ \PP_{\! e} \, \cU_0 $ as in  \eqref{defPU0:2simpli} into \eqref{modu0effective}. Recall that
\begin{equation}
\label{eqn:CMV}
u_0 = \frac{\sqrt{m_e}\, u_{e0} + \sqrt{m_i}\, u_{i0}}{\sqrt{\underline n}\,(m_e+m_i)} = \frac{m_e \, v_{e0} + m_i \, v_{i0}}{m_e+m_i} \, ,
\end{equation}
corresponds (in terms of the fluid velocities $ v_e $ and $ v_i $) to the velocity of the center of mass.

\begin{lem}[Computation of $ \PP_{\! e} \cF(0,\PP_{\! e} \,\cU_0) $]
  \label{compucf000}
  We find 
  \begin{equation}\label{pecfpe}
    \quad \PP_{\! e} \cF(0,\PP_{\! e} \,\cU_0) =
    \left(\begin{array}{c}
      0\\0\\ \displaystyle \frac{\sqrt{{\underline n} m_e}}{m_e m_i} \, {\underline b} \ \LL \bigl( (\nabla \times B_0) \times B_0 \bigr) + \sqrt{\frac{\underline{n}}{m_e}} \,  \frac{\underline{b} \, \Delta}{1 - \underline{b} \, \Delta} \, \LL (D_0 \times B_0) \medskip \\
      \displaystyle \frac{\sqrt{{\underline n} m_i}}{m_e m_i} \, {\underline b} \ \LL \bigl( (\nabla \times B_0) \times B_0 \bigr) -  \sqrt{\frac{\underline{n}}{m_i}} \,  \frac{\underline{b}\, \Delta}{1 - \underline{b} \, \Delta} \, \LL (D_0 \times B_0)\medskip\\ 0 \\
      \displaystyle (1-\underline{b}\, \Delta)^{-1}\, \nabla \times (D_0 \times B_0)
  \end{array} \right) \,,
  \end{equation}
  where, with $\updelta$ and $\myunderbar{\rho}$  given by \eqref{eqn:cteXMHD} with $Z=1$, we have introduced the auxiliary expression
  \begin{equation}
    \label{computedeff+}
    \qquad D_0 := \LL u_0 - (1-\updelta)\, (m_i / \myunderbar{\rho})  \, \nabla \times B_0 \, .
  \end{equation}
\end{lem}

\begin{proof}
  Since $ \cU_0 = \PP_{\! e} \,\cU_0 $, from Corollary \ref{furthersimplification}, we have
  \begin{equation}
    \label{reducedppe}
    \PP_{\! e} \, \cU_0 = {}^t \bigl(
    \bar{\varrho}_{e0}, \bar{\varrho}_{i0},
    \sqrt{\underline n} \, {}^t \tilde u_{e0}, \sqrt{\underline n} \, {}^t \tilde u_{i0}, 0 , {}^t \LL B_0 \bigr) \, .
  \end{equation}
  with
  \begin{equation}
    \label{lesdeuxus}
    \qquad \qquad \tilde {u}_{e0} := \sqrt{m_e}\, \LL u_0 - \frac{\underline{b}}{\sqrt {m_e}} \, \nabla \times B_0\,,\qquad
    \tilde {u}_{i0} := \sqrt{m_i} \, \LL u_0 + \frac{\underline{b}}{\sqrt {m_i}} \, \nabla \times B_0\,.
  \end{equation}
  From there, coming back to \eqref{remainsspecify}, we find
  \begin{equation}
    \label{devdecF}
    \quad \cF(0,\PP_{\! e} \,\cU_0) = {}^t \Bigl( 0 , 0 , - \sqrt{\underline n} \, {}^t \bigl(\tilde u_{e0} \times ( \LL B_0/ m_e ) \bigr),
    \sqrt{\underline n} \, {}^t \bigl( \tilde u_{i0} \times (\LL B_0 / m_i) \bigr) , \cF_E , 0\Bigr) \,.
  \end{equation}
  There is no guarantee that  $ \cF(0,\PP_{\! e} \,\cU_0) $ is polarized inside $ \cH $. We cannot  employ \eqref{defPU0:2simpli}. Instead, we have to implement \eqref{defPU0:2} which leads to
  \[
  \PP_{\! e} \cF(0,\PP_{\! e} \,\cU_0) =
  \left(\begin{array}{c}
    0\\0\\ \displaystyle \frac{\sqrt{{\underline n} m_e}}{m_e + m_i} \ \LL ( \widetilde S_0 \times B_0 ) + \sqrt{\frac{\underline{n}}{m_e}} \,  \frac{\underline{b} \, \Delta}{1 - \underline{b} \, \Delta} \, \LL (\widetilde D_0 \times B_0) \medskip \\
    \displaystyle \frac{\sqrt{{\underline n} m_i}}{m_e + m_i} \ \LL ( \widetilde S_0 \times B_0 ) -  \sqrt{\frac{\underline{n}}{m_i}} \,  \frac{\underline{b}\, \Delta}{1 - \underline{b} \, \Delta} \, \LL (\widetilde D_0 \times B_0)\medskip\\ 0 \\
    \displaystyle (1-\underline{b}\, \Delta)^{-1}\, \nabla \times (\widetilde D_0 \times B_0)
  \end{array} \right)\, ,
  \]
  with
  \[
  \widetilde S_0 := \frac{\tilde{u}_{i0}}{\sqrt{m_i}} - \frac{\tilde{u}_{e0}}{\sqrt{m_e}} \, ,
  \qquad \widetilde D_0 := \underline{n} \, \underline{b} \ \bigg( \frac{\tilde{u}_{i0}}{m_i \, \sqrt{m_i}} + \frac{\tilde{u}_{e0}} {m_e \, \sqrt{m_e}} \, \bigg) \,.
  \]
  From \eqref{lesdeuxus}, we obtain
  \[
  \begin{array}{l}
    \displaystyle \widetilde S_0 = \underline{b} \ \Bigl(\frac{1}{m_e} + \frac{1}{m_e} \Bigr) \
    \nabla \times B_0 = \frac{m_e +m_i}{m_e \, m_i} \, {\underline b} \ \nabla \times B_0 \, , \smallskip \\
    \displaystyle \widetilde D_0 = \LL u_0 + {\underline n} \, {\underline b}^2 \ \Bigl(\frac{1}{m_i^2} - \frac{1}{m_e^2} \Bigr) \ \nabla \times B_0 = D_0 \, .
  \end{array}
  \]
  After substitution, we recognize \eqref{pecfpe}.
\end{proof}

\begin{lem}[Computation of the quasilinear part inside ESLM]
  \label{computatt}
  We find
  \begin{equation}
    \label{pecfpebis}
    \qquad \sum_{j=1}^3 \PP_{\! e} \cA_j(0,\PP_{\! e}\,\cU_0) \, \part_{x_j} \PP_{\! e}\,\cU_0 = -  \left(\begin{array}{c}
      0 \\0 \\
      \displaystyle \sqrt{\underline{n} m_e} \,\LL \nabla \cdot S_0 + \sqrt{\frac{\underline{n}}{m_e}} \, \frac{\underline{b} \, \Delta}{1-\underline{b}\, \Delta} \ \LL T_0 \medskip \\
      \displaystyle \sqrt{\underline{n} m_i} \,\LL \nabla \cdot S_0 -  \sqrt{\frac{\underline{n}}{m_i}} \, \frac{\underline{b} \, \Delta}{1-\underline{b}\, \Delta} \ \LL T_0\medskip \\
      0\\
      \displaystyle   (1-\underline{b}\, \Delta)^{-1}\, \nabla \times T_0
    \end{array} \right) \, ,
  \end{equation}
  where
  \begin{equation}
    \label{deddds}
    S_0 := \LL u_0 \otimes \LL u_0 + \myunderbar{\rho}^{-1} \, \underline{b} \  (\nabla \times B_0) \otimes (\nabla \times B_0) \, ,
  \end{equation}
  and
  \begin{equation}
    \label{defffT}
    \begin{array}{rl}
      T_0 = \! \! \! & - \,  \underline{b} \ \LL u_0 \times \bigl( \nabla \times (\nabla \times B_0) \bigr) - \underline{b} \,  (\nabla \times B_0) \times ( \nabla \times \LL u_0) \smallskip \\
   \ & \displaystyle + \,
   \myunderbar{\rho}^{-2} \, m_i^3 \, \updelta \, (1-\updelta) \,  (\nabla \times B_0) \times \bigl( \nabla \times  (\nabla \times B_0) \bigr) \,.
    \end{array}
  \end{equation}
\end{lem}

\begin{proof}
  The starting point is \eqref{reducedppe} and \eqref{lesdeuxus}. From the definition of the matrices $ \cA_j (0,\cdot) $, since both vector fields $ \tilde u_{e0} $ and $ \tilde u_{i0} $ are divergence free, we obtain
  \begin{equation}
    \label{devdecA}
    \sum_{j=1}^3 \cA_j(0,\PP_{\! e}\,\cU_0) \, \part_{x_j} \PP_{\! e} \,\cU_0  =
    - \left(\begin{array}{c}
      0\\0\\ \displaystyle
      \sqrt{\frac{\underline{n}}{m_e}} \ \nabla \cdot (\tilde {u}_{e0} \otimes \tilde {u}_{e0}) \medskip \\
      \displaystyle \sqrt{\frac{\underline{n}}{m_i}} \ \nabla \cdot (\tilde {u}_{i0} \otimes \tilde {u}_{i0}) \medskip  \\  0\\ 0
    \end{array} \right) \,.
  \end{equation}
  We compose \eqref{devdecA} on the left with $ \PP_{\! e} $. To this end, we  implement  \eqref{defPU0:2}. This furnishes
  \[
  \sum_{j=1}^3 \PP_{\! e} \cA_j(0,\PP_{\! e}\,\cU_0) \, \part_{x_j} \PP_{\! e}\,\cU_0 = -\PP_{\! e}
  \left(\begin{array}{c}
    0\\
    0\\
    \displaystyle
    \sqrt{\underline{n} \,  m_e} \ \LL (\nabla \cdot \wideparen S_0) + \sqrt{\frac{\underline{n}}{m_e}} \,
    \frac{\underline{b}^2 \, \Delta}{1-\underline{b}\, \Delta} \ \LL (\nabla \cdot \wideparen D_0) \medskip \\
    \displaystyle
    \sqrt{\underline{n} \,  m_i} \ \LL (\nabla \cdot \wideparen S_0) - \sqrt{\frac{\underline{n}}{m_i}} \,
    \frac{\underline{b}^2 \, \Delta}{1-\underline{b}\, \Delta} \ \LL (\nabla \cdot \wideparen D_0) \medskip  \\  0\\
    \displaystyle  \frac{\underline{b}}{1-\underline{b}\, \Delta} \ \nabla \times (\nabla \cdot \wideparen D_0 )
  \end{array} \right) \,,
  \]
  with
  \[
  \wideparen S_0 := \frac{1}{m_e + m_i} \ (\tilde {u}_{e0} \otimes \tilde {u}_{e0} + \tilde {u}_{i0} \otimes \tilde {u}_{i0}) \, ,
  \qquad \wideparen D_0 := \frac{\underline n}{m_i} \, \tilde {u}_{i0} \otimes \tilde {u}_{i0} -  \frac{\underline n}{m_e} \, \tilde {u}_{e0} \otimes \tilde {u}_{e0} \,.
  \]
  We can substitute the values of $ \tilde u_{s0} $ given by \eqref{lesdeuxus} inside $ \wideparen S_0 $ to get $ \wideparen S_0 = S_0 $ with $ S_0 $ as in \eqref{deddds}. Moreover, exploiting the vector-dot-Del identity
  \begin{equation}
    \label{vector-dot-Del}
    (\tilde u_s \cdot \nabla) \tilde u_s = \frac{1}{2} \ \nabla \vert \tilde u_s \vert^2 - \tilde  u_s \times (\nabla \times \tilde u_s) \, , \qquad s \in \{e,i \} \, ,
  \end{equation}
  we obtain
  \begin{align*}
    & \nabla \cdot \wideparen D_0 = {\underline n} \ \nabla \cdot \bigg(
    \frac{\tilde {u}_i\otimes \tilde {u}_i}{ m_i}
    - \frac{\tilde {u}_e\otimes \tilde {u}_e}{m_e}\bigg) = {\underline n} \ \bigg(
    \frac{(\tilde {u}_i \cdot  \nabla )  \tilde {u}_i}{ m_i}
    - \frac{(\tilde {u}_e \cdot   \nabla   \tilde {u}_e)}{m_e}\bigg) \\
    & \qquad = \, \frac{{\underline n}}{2} \, \nabla \bigg(
    \frac{\vert \tilde {u}_i \vert^2}{ m_i}
    -
    \frac{\vert \tilde {u}_e \vert^2}{ m_e}\bigg) - \frac{{\underline n}}{m_i} \, \tilde {u}_i \times (\nabla \times \tilde {u}_i) +  \frac{{\underline n}}{m_e} \, \tilde {u}_e \times (\nabla \times \tilde {u}_e) \nonumber \\
    &\qquad = \frac{{\underline n}}{2} \, \nabla \bigg(
    \frac{\vert \tilde {u}_i \vert^2}{ m_i} -
    \frac{\vert \tilde {u}_e \vert^2}{ m_e}\bigg) - \LL u_0 \times \bigl( \nabla \times (\nabla \times B_0) \bigr) -  (\nabla \times B_0) \times ( \nabla \times \LL u_0) \\
    &  \qquad \quad \ + \, \frac{m_i^3} {{\underline b}\, \myunderbar{\rho}^2}\, \updelta \, (1-\updelta) \,  (\nabla \times B_0) \times \bigl( \nabla \times  (\nabla \times B_0) \bigr) \,.
  \end{align*}
  The gradient disappears under the action of $ \nabla \times $. It follows that $ {\underline b} \, \nabla \times (\nabla \cdot \wideparen D_0) = \nabla \times T_0 $ with $ T_0 $ as in \eqref{defffT}. From there, we have also
  \[
    {\underline b} \, \Delta  \LL (\nabla \cdot \wideparen D_0) = - \underline{b} \,
    \nabla \times \nabla \times  (\nabla \cdot \wideparen D_0) = - \nabla \times \nabla \times T_0 = \Delta \LL T_0 \,.
  \]
  After substitution, we recover the content of Lemma \ref{computatt}.
\end{proof}


\subsubsection{Comparison of ESLM with SLM and XMHD}
\label{subs:end} SLM is meaningful when $ \cU_0^0 = \PP \, \cU_0^0 $. It can also be considered under the more restrictive condition $ \cU_0^0 = \PP_{\! e} \, \cU_0^0 $. Then, looking at SLM comes down to solving ESLM.

\begin{cor}[Equivalence between ESLM and SLM]
  \label{equiequi}
  Assume that the initial data $ \cU_0^0 $ is polarized inside $ \cH $. Then, the solution to \eqref{modu0effective} is also a solution to \eqref{modu0prepa2}.
\end{cor}

\begin{proof}
  Let $ \cU_0 = \PP_{\! e} \,  \cU_0 $ be the solution to \eqref{modu0effective}. We must prove that this $ \cU_0 $ satisfies \eqref{modu0prepa2}. Since $ \PP \, \PP_{\! e} = \PP_{\! e} $, by substracting \eqref{modu0effective} from  \eqref{modu0prepa2}, it suffices to show that
  \[
  \sum_{j=1}^3 \, (\PP - \PP_{\! e}) \,  \cA_j(0,\PP_{\! e}\,\cU_0) \, \part_{x_j} \PP_{\! e} \, \cU_0 + (\PP - \PP_{\! e}) \cF(0,\PP_{\! e} \,\cU_0) = 0 \, .
  \]
  In view of \eqref{devdecF} and \eqref{devdecA}, this property is a direct consequence of Lemma \ref{actions} on condition that $ \nabla \cdot \cF_E (0,\PP_{\! e} \, \cU_0) = 0 $. Recall from \eqref{remainsspecify} and \eqref{reducedppe} that
  \[
  \cF_E (0,\PP_{\! e} \, \cU_0) = \myunderbar{\mathscr R} (g_e^{-1})\Big(0, \frac{ \bar{\varrho}_{e0}}{\sqrt{{\underline n} \,  m_e}} \Big)
  \ \frac{\tilde u_{e0}}{\sqrt{m_e}}  - \myunderbar{\mathscr R} (
  g_i^{-1}) \Big( 0 ,\frac{ \bar{\varrho}_{i\varepsilon}}{\sqrt{{\underline n} \,  m_i}} \Big) \ \frac{\tilde u_{i0}}{\sqrt{m_i}} \,.
  \]
  Through \eqref{mathscrR},  \eqref{threeconstants} and \eqref{lesdeuxus}, this can be further simplified into
  \[
  \cF_E (0,\PP_{\! e} \, \cU_0) := \bar{\varrho} \ \sqrt{\frac{\underline n}{{\underline p}'_i}}
  \ \bigg( \frac{\tilde u_{e0}}{\sqrt{m_e}} - \frac{\tilde u_{i0}}{\sqrt{m_i}} \bigg)
  = -  \frac{\bar{\varrho}}{\sqrt{{\underline n} \, {\underline p}'_i}} \ \nabla \times B_0 \,.
  \]
  It becomes apparent that $ \nabla \cdot \cF_E (0,\PP_{\! e} \, \cU_0) = 0 $, as expected.
\end{proof}

We now examine more precisely the different parts of  ESLM \eqref{modu0effective}. The component $ E_0 $ is absent from $ \PP_{\! e} \, \cU_0 $. Moreover, gathering \eqref{reducedppe}, Lemma \ref{compucf000} and Lemma \ref{computatt}, we find $ \part_t \bar{\varrho}_{s0} = 0 $. Thus, we can retain
\begin{equation*}
  \label{constantvarrho1}
  \qquad \varrho_{s0} (t) = \bar{\varrho}_{s0} (t) = \bar{\varrho}_{s0}^0 \, ,
  \qquad \forall \, t \in \mathbb R_+ \, , \qquad \forall \, s \in \{ e, i \} \, .
\end{equation*}
From \eqref{passagecUU0}, we immediately deduce 
\begin{equation*}
  \label{constantvarrho2}
  \qquad n_{s0} (t) = \overline{n}_{s0} (t) = \overline{n}_{s0}^0 = \sqrt{{\underline n}/{\underline p}'_s} \
  \bar{\varrho}^0_{s0} \, ,
  \qquad \forall \, t \in \mathbb R_+ \, , \qquad \forall \, s \in \{ e, i \} \, .
\end{equation*}
Recall that we must also have \eqref{threeconstants} so that (for $ Z=1 $)
\begin{equation*}
  n_{i0}(t,x) = n_{e0}(t,x) = \overline{n}_{i0}^0 = \overline{n}_{e0}^0 = \overline{n} \, ,
  \qquad \overline{n} := \sqrt{\underline n / {\underline p}'_i }\, \bar{\varrho} \, .
\end{equation*}
At this stage, we have already checked  \eqref{whichissuch}.

Now, we consider the velocity components. We sum the third vector components of  \eqref{defPU0:2simpli}, \eqref{pecfpe} and \eqref{pecfpebis}  multiplied by the weight $\sqrt{m_e/\underline{n}}/(m_e+m_i)$ to the fourth  vector components of respectively  \eqref{defPU0:2simpli}, \eqref{pecfpe} and \eqref{pecfpebis}  multiplied by $\sqrt{m_i/\underline{n}}/(m_e+m_i)$. By this way, we obtain
\begin{equation}
  \label{eqn:preMHD1}
  \begin{array}{rl} \displaystyle \partial_t (\LL u_0) + \LL \nabla \cdot (\LL u_0 \otimes \LL u_0) \! \! \! &
    \displaystyle + \, \updelta \, \frac{m_i^2 }{\myunderbar{\rho}^2 } \,
  \LL \nabla \cdot \big( (\nabla \times B_0) \otimes (\nabla \times B_0)\big) \\
  \ & \displaystyle - \,  \frac{1}{\myunderbar{\rho}} \,  \LL \big( (\nabla \times B_0) \times B_0\big)=0\,.
  \end{array}
\end{equation}
Using $\nabla \cdot B_0=0$ and \eqref{vector-dot-Del}, we get
\begin{equation}
  \label{simplfallait} \LL \nabla \cdot \big( (\nabla \times B_0) \otimes (\nabla \times B_0)\big) = \LL  \big((\nabla \times B_0) \times \Delta B_0) \big) \, .
\end{equation}
In coherence with \eqref{lienBB*2}, let us define
\begin{equation}
  \label{defB0star}
  B_0^\ast:=B_0 + (\updelta \, m_i^2 /\myunderbar{\rho}) \ \nabla \times (\nabla \times B_0) = (1- \underline{b} \, \Delta )B_0\, .
\end{equation}
With \eqref{simplfallait} and \eqref{defB0star}, the equation \eqref{eqn:preMHD1} can be recast into
\begin{equation}
  \label{eqn:preMHD2}
  \partial_t (\LL u_0) + \LL \nabla \cdot \big(\LL u_0 \otimes \LL u_0 \big) +
  \frac{1}{\myunderbar{\rho}} \,  \LL \big(B_0^\ast \times (\nabla \times B_0) \big)=0\,.
\end{equation}
The equation \eqref{eqn:preMHD2} is equivalent to the existence a scalar (pressure) function $p_0$ (a Lagrange multiplier) such that
\begin{equation}
  \label{eqn:preMHD3}
  \partial_t  u_0 + u_0 \cdot \nabla u_0 + \frac{1}{\myunderbar{\rho}}\,B_0^\ast \times (\nabla \times B_0) + \nabla p_0=0\,,  \qquad \nabla \cdot u_0=0\,.
\end{equation}
We turn now to the last vector components of  \eqref{defPU0:2simpli}, \eqref{pecfpe} and \eqref{pecfpebis}. After summation and applying the operator $(1-\underline{b}\,\Delta)$ to the sum, we obtain
\[
\part_t (\LL B_0^*) + \nabla \times (B_0 \times D_0 + T_0) = 0 \, ,
\]
which is the same as
\begin{equation}
  \label{eqn:preMHD67}
  \part_t B_0^* + \nabla \times (B_0 \times D_0 + T_0) = 0 \, , \qquad \nabla \cdot B_0^* = 0 \, .
\end{equation}
With the help of \eqref{computedeff+} and \eqref{defffT}, we can compute
\[
\begin{array}{rl}
  B_0 \times D_0 + T_0 = \! \! \! & \ B_0 \times \LL u_0 - (1-\updelta)\, (m_i / \myunderbar{\rho})  \, B_0 \times (\nabla \times B_0) \smallskip \\
  \ & - \underline{b} \ \LL u_0 \times \bigl( \nabla \times (\nabla \times B_0) \bigr) - \underline{b} \,  (\nabla \times B_0) \times ( \nabla \times \LL u_0) \smallskip \\
  \ & + \,
  \myunderbar{\rho}^{-2} \, m_i^3 \, \updelta \, (1-\updelta) \,  (\nabla \times B_0) \times \bigl( \nabla \times  (\nabla \times B_0) \bigr) \,.
\end{array}
\]
In the right-hand side, we can put together the first and third terms, as well as the second and last terms. Then, exploiting \eqref{defB0star}, there remains
\[
B_0 \times D_0 + T_0 = B_0^* \times \bigl(\LL u_0 - (1-\updelta)\, (m_i / \myunderbar{\rho})
\,  \nabla \times B_0 \bigr) + \underline{b} \,   ( \nabla \times \LL u_0) \times (\nabla \times B_0) \,.
\]
The three equations \eqref{defB0star}, \eqref{eqn:preMHD3} and \eqref{eqn:preMHD67} completed with $ B_0 \times D_0 + T_0 $ as above are exactly the same as \eqref{lienBB*2}, \eqref{divfreeini} and \eqref{syssimpli}. This means that $ (u_0,B_0) $ is indeed a solution to incompressible XMHD.
The analysis of ESLM  is not yet over. At this stage, we have extracted XMHD from ESLM. But, if we put aside the evolution of $ \varrho_{s0} $ which is  trivial, when dealing with $ \PP_{\! e} \, \cU_0 $, we have to focus on $ 9 $ equations on $ \bigl( u_{e0},u_{i0},  B_0 \bigr) \in \RR^9 $, while $ (u_0,B_0) \in \RR^6 $. But that is without counting on \eqref{Ampereslaw} which connects the velocities $ u_{e0} $ and $ u_{i0} $ to $ B_0 $. This implies that $ \PP_{\! e} \, \cU_0 $ involves in practice a lower number of six state variables, say $ (u_0,B_0) $.  However, in doing so, there are difficulties.

The first problem is that the combination (as vectors in $ \RR^{14} $) of \eqref{defPU0:2simpli}, \eqref{pecfpe} and \eqref{pecfpebis} in order to recover \eqref{modu0effective} could seem overdetermined with its $ 9 $ equations for its six remaining unknonws. This is  a little misleading. ESLM is by construction well-posed with a correct number of unknowns and equations if we take into account everywhere (including in the formulation of equations) the impact of the projection $ \PP_e $.

As a matter of fact, some equations issued from the juxtaposition of  \eqref{defPU0:2simpli}, \eqref{pecfpe} and \eqref{pecfpebis} are redundant.
 If we look at the difference\footnote{This operation is reminiscent of the derivation of the generalized Ohm's law (in Paragraph \ref{subsub:ohm}). Except that the filtering procedure does replace the electric filed by the Maxwell--Faraday equation, which leads to a repetition of \eqref{eqn:preMHD67}. } between the third vector components of  \eqref{defPU0:2simpli}, \eqref{pecfpe} and \eqref{pecfpebis}  multiplied by the weight $\sqrt{\underline{n}/ m_i} $ and the fourth  vector components of respectively  \eqref{defPU0:2simpli}, \eqref{pecfpe} and \eqref{pecfpebis}  multiplied by $  \sqrt{\underline{n}/ m_e}$, we obtain
\[
\nabla \times \bigl \lbrack \part_t B_0^* + \nabla \times (B_0 \times D_0 + T_0) \bigr \rbrack = 0 \, , \qquad \nabla \cdot B_0^* = 0 \, ,
\]
which is clearly induced by \eqref{eqn:preMHD67}.

The second trouble is that the use of $ (u_0,B_0) $ destroys the organisation of ESLM. It is clear that  \eqref{modu0effective} is a quasilinear symmetric system. Since $ \cU_0 = \PP_{\! e} \,\cU_0 $, we can exploit \eqref{Ampereslaw} to replace everywhere in Lemmas \ref{compucf000} and   \ref{computatt} the expression $ \nabla \times B_0 $ by $ J_0 $. From this perspective, the nonlinear pseudo-differential operators on $ \cU_0 $  implied at the level of \eqref{pecfpe} and \eqref{pecfpebis} are respectively of order $ 0 $ and $ 1 $.
But,  expressed in terms of $ (u_0,B_0) $, the lines \eqref{pecfpe} and \eqref{pecfpebis} involve operators of higher  order $ 1 $ and $ 2 $. In other words, due to the dispersive relation \eqref{Ampereslaw}, resorting to the reduced set of variables $ (u_0,B_0) $, changes the order of derivatives. The hyperbolic status of XMHD (involving $ u_0 $ and $ B_0 $) becomes less accessible, and the well-posedness of XMHD is a quite  complicated issue  \cite{BC1}.

The equivalence between SLM and XMHD is now clear. From Corollary \ref{equiequi}, SLM reduces to ESLM (at least for the data under consideration). Moreover, there exists a complete correspondance between the solutions $\PP_{\! e} \,\cU_0 = {}^t(\bar{\varrho}_{e0}^0, \bar{\varrho}^0_{i0}, {}^t u_{e0}, {}^t u_{i0},0, {}^t B_0 )$ of ESLM and the
solutions $(u_0,B_0^*)$ to incompressible XMHD. As a matter of fact, we can pass from $ (u_{e0}, u_{i0},B_0 )$ to $(u_0,B_0^*)$ by using \eqref{eqn:CMV} and \eqref{defB0star}. In the opposite direction, we can deduce $ (u_{e0}, u_{i0},B_0 )$ from $(u_0,B_0^*)$ through the formulas \eqref{defB0star} and
\begin{align*}
  &u_{e0} =\sqrt{\underline{n} \,  m_e} \ u_0 -\underline{b}\ \sqrt{\underline{n}/m_e} \ \nabla \times
  (\id- \underline{b}\, \Delta)^{-1}B_0^\ast\,, \\
  & u_{i0}  =\sqrt{\underline{n} \,  m_i} \ u_0 +\underline{b} \  \sqrt{\underline{n}/m_i} \ \,  \nabla \times(\id- \underline{b}\, \Delta)^{-1}B_0^\ast\, .
\end{align*}
Coming back through \eqref{passagecUU0} to the initial state variables, we recover \eqref{importantrelation}. This concludes the proof of Theorem \ref{maintheo}.


\subsection{Other situations}
\label{Othersituations}
The purpose here is to complement the preceding discussion with some information about situations that fall more or less directly under the preceding rules.


\subsubsection{The case of a background with zero density}
\label{subs:zerodensity} The description of dilute plasmas can be achieved only by imposing $ \underline n = 0 $ (or $ \underline U = 0 $). The discussion depends heavily on the choice of pressure laws. As in \cite{GIP}, we work here with $ p_s ( {\rm n}_s) = P_s \ {\rm n}_s^2 / 2 $ where $ s \in \{e,i \} $ and  $  P_s \in \RR_+^* $.
Define $ \iota_e := -1 $ and $ \iota_i = 1 $. Then, we  change $ {\rm n}_s $ into $ \tilde {\rm n}_s := \sqrt{{\rm n}_s} $. By  developing the solution to
\eqref{TFEM} expressed in terms of $ {}^t (\tilde {\rm n}_{e\eps},\tilde {\rm n}_{i\eps},{}^t \tilde {\rm v}_{e\eps}, {}^t \tilde {\rm v}_{i\eps}, {}^t  \widetilde {\rm E}_\eps,  {}^t \widetilde{\rm B}_\eps)$\footnote{with the trivial change of unknowns $\tilde {\rm v}_{s\eps} = {\rm v}_{s\eps}$, for $s\in\{e,i\}$, $\widetilde {\rm E}_\eps = {\rm E}_\eps$, and $\widetilde{\rm B}_\eps={\rm B}_\eps$.} as indicated in \eqref{linearalo0}\footnote{This means that $ \tilde {\rm n}_{s\eps} = \eps \, q_{s\eps} $ so that $ {\rm n}_{s\eps} = \cO(\eps^2) $, whereas the other components $ \tilde {\rm v}_{s\eps} = \eps \, v_{s\eps}$, $ \widetilde {\rm E}_{\eps} = \eps \, E_\eps $ and $  \widetilde {\rm B}_{\eps} = \eps \, B_\eps $ are still $ \cO(\eps) $.}, the issue is to investigate the singular system
\[
\left \lbrace \begin{array}{ll}
\displaystyle 2 \, P_s \ \bigl (\part_t q_{s\eps} + (v_{s\eps}\cdot \nabla) q_{s\eps}\bigr ) + P_s \ q_{s\eps} \ \nabla \cdot v_{s\eps} = 0 \, , & s \in \{e,i\} \, , \smallskip \\
\displaystyle \frac{m_s}{2} \, \bigl ( \part_t v_{s\eps} + (v_{s\eps}\cdot \nabla) v_{s\eps}\bigr )+ P_s \ q_{s\eps} \ \nabla  q_{s\eps} = \frac{\iota_s}{2 \, \eps} \, E_\eps + \frac{\iota_s}{2} \, v_{s\eps} \times B_\eps \, , \quad & s \in \{e,i\} \, , \smallskip \\
\displaystyle \part_t E_\eps - \frac{1}{\eps} \ \nabla \times B_\eps = \eps \ (q^2_{e\eps} \, v_{e\eps} - q^2_{i\eps} \, v_{i\eps} ) \, , \smallskip\\
\displaystyle \part_t B_\eps + \frac{1}{\eps} \ \nabla \times E_\eps = 0 \, ,
\end{array}
\right. \]
together with $
\nabla \cdot B_\eps = 0 $ and $ \nabla \cdot E_\eps + \eps \, (q^2_{e\eps} - q^2_{i\eps})= 0 $.
Again, the asymptotic analysis of these equations falls under the scope of \cite{S07}. It is easy to see that the conditions $ E_0 = 0 $ and $ B_0 = 0 $ are the substitute\footnote{Note that the two conditions $ E_0 = 0 $ and $ B_0 = 0 $ are not required in \cite{GIP}. In fact, the situation examined in \cite{GIP} is within the FL model.} for $ \PP_{\! e} \, \cU_0 = \cU_0 $. Starting from there, the Slow Limit Model corresponding to the above framework is just made of a trivial electromagnetic field together with the following two independent compressible equations ($ s \in \{e,i\} $)
\[
\left \lbrace \begin{array}{l}
\displaystyle 2 \, P_s \ \bigl ( \part_t q_{s0} + (v_{s0}\cdot \nabla) q_{s0}\bigr ) + P_s \ q_{s0} \ \nabla \cdot v_{s0} = 0 \, , \smallskip \\
\displaystyle \frac{m_s}{2} \, \bigl (\part_t v_{s0} + (v_{s0}\cdot \nabla) v_{s0}\bigr ) + P_s \ q_{s0} \ \nabla  q_{s0} = 0 \, .
\end{array}
\right.
\]
For unprepared data (when $ E_0^0 \not \equiv 0 $ or $ B_0^0 \not \equiv 0 $), time oscillations on the electromagnetic field do persist. The Fast Limit Model contains the free Maxwell equations
\[
\left \lbrace \begin{array}{lll}
\displaystyle \part_\tau E_0^r - \nabla \times B^r_0 = 0 \, , \quad & \nabla \cdot E^r_0 = 0 \, , \quad & {E^r_0}_{\mid \tau = 0} = E^0_0 \, , \smallskip\\
\displaystyle \part_\tau B_0^r + \nabla \times E^r_0 = 0 \, , \quad & \nabla \cdot B^r_0 = 0 \, , \quad & {B^r_0}_{\mid \tau = 0} = B^0_0 \, .
\end{array}
\right.
\]


\subsubsection{The case of irrotational flows}
\label{subs:irroflows}
This is when the condition \eqref{Pausadercdt} is imposed at the initial time $ t = 0 $, and therefore it holds at any time $ t \in \RR_+^* $.  As a consequence, passing to the limit through \eqref{difflim}, we can assert that
\begin{equation}\label{PauPau}
    \delta \, \nabla \times v_{e0} = - \nabla \times v_{i0} = B_0 \, .
\end{equation}
As before, we consider separately the two situations $ \underline n \not = 0 $ and $ \underline n = 0 $. 

$ \bullet $ First, assume that $ \underline n \not = 0 $. Then, we must have \eqref{Ampereslaw}, which is the same as
\begin{equation*}
  \label{reapeat}
 \nabla \times B_0 = \underline{n} \ ( v_{i0} -v_{e0} ) \, .
\end{equation*}
It follows that
\[
\nabla \times (\nabla \times B_0) = - \Delta B_0 = \underline{n} \ ( \nabla \times v_{i0} - \nabla \times v_{e0} ) = -  \underline{n} \ (1+\updelta) \ B_0 / \updelta \, .
\]
On the Fourier side, this is the same as
\[
\bigl( \vert k \vert^2 + \underline{n} \ (1+\updelta) / \updelta \bigr) \ B_{0k} = 0 \, , \qquad \forall k \in \ZZ^3 \, ,
\]
so that $ B_0 \equiv 0 $. Coming back to \eqref{PauPau}, since both $ v_{e0} $ and $ v_{i0} $ are divergence free, we must also have $ v_{e0} (t,x) = \overline v_{e0} (t) $ and $ v_{e0} (t,x) = \overline v_{i0} (t) $. This means that $ u_0 (t,x) = \overline u_0 (t) $, while  XMHD reduces to $ \part_t \overline u_0 = 0 $, so that $ u_0(t,x) = \overline u_0^0 $. At the end, there remains a constant solution $ U_0 (t,x) = \overline {U}^0_0 $. In the perspective of our asymptotic analysis, for amplitudes of size $ \cO(\eps) $ and after a time $ \tau =\cO(1/\eps) $, the irrotational condition \eqref{PauPau} kills any form of non trivial  evolution, at least at the main $ \cO(\eps) $-order.

 $ \bullet $ Secondly, assume that $ \underline n = 0 $. From the relation \eqref{PauPau}, we can infer  that
$ \nabla \times v_{e0} = 0 $ and $ \nabla \times v_{i0} = 0 $. Coming back to Paragraph \ref{subs:zerodensity}, The Slow Limit Model reduces to two irrotational (potential) compressible flows.



\appendix

\section{Justification of the passage to the limit}
\label{appendix:ME}
Here, we justify the  modulation equation \eqref{modu0}-\eqref{modu0ini0} and the convergence result \eqref{difflim}. The proof is inspired from the filtering unitary group method introduced by S. Schochet, see the survey articles \cite{Ala08,Ga,S07}, in continuation of \cite{BM61,JJLX,KM,K,JSX19,LM,MS, SV07, Sch94a,Sch94b}. Consider the expression $ U_\eps (t,x) $ defined by \eqref{defUepst}. We start with the extraction of \eqref{difflim}. Using standard energy estimates, valid in the  general case (prepared or not), we deduce that there exists a time $T \in \RR_+^* $ and a constant $C$  independent of $\varepsilon$ such that, given any $ s>5/2 $, we have
\begin{equation}
  \label{AME:estuniU}
  \|U_{\varepsilon}\|_{L^\infty([0,T];H^s)} \leq C \, , \qquad \forall \, \eps \in ]0,\eps_0 ] \, .
\end{equation}
This bound relies on the symmetry of \eqref{HSeps} and the fact that the singular operator $\varepsilon^{-1} \cL$ does not contribute to energy estimates since it is skew-adjoint and has constant coefficients.

In the  unprepared case, we cannot have an estimate for $\partial_t U_\varepsilon$ in $L^\infty([0,T];H^s)$ uniformly in $\varepsilon$, since $\partial_t  U_\varepsilon$ may not be bounded uniformly in $\varepsilon$ at time zero. Fortunately, using a convenient change of unknowns, it turns out that the $\mathcal{O}(1/\varepsilon)$ part in \eqref{HSeps} and thus in $\partial_t U_\varepsilon$ can be cancelled as follows. The new unknown $V_\varepsilon$ is defined by filtering out the fast time oscillations of $U_\varepsilon$ as
\begin{equation}
  \label{AME:CV}
  V_\varepsilon(t,x)= \cS(-t/\varepsilon) \, U_\varepsilon(t,x)\,, \qquad \cS(\tau):= e^{\tau \mathcal L} \, ,
\end{equation}
where $\{  \cS(\tau)\, ; \, \tau \in \RR\}$ is a one-parameter group of isometry in $H^s$ which is generated by the skew-adjoint linear operator $\mathcal{L}=\mathcal{L}(\myunderbar{U})$. Since the group $\cS$ is an isometry in $H^s$ ($\forall s\geq 0$), using  \eqref{AME:estuniU}, we obtain from \eqref{AME:CV} the uniform control
\begin{equation}
  \label{AME:estuniV}
  \|V_{\varepsilon}\|_{L^\infty([0,T];H^s)} \leq C \, , \qquad \forall \, \eps \in ]0,\eps_0 ] \, .
\end{equation}
Futhermore, using \eqref{HSeps} we obtain
\begin{equation}
  \label{AME:PTV}
  \partial_tV_\varepsilon(t,x)= \cS(-t/\varepsilon)\big(A(\varepsilon,\myunderbar{U},U_\varepsilon, D_x)U_\varepsilon +f(\varepsilon,\myunderbar{U},U_\varepsilon)\big)\,.
\end{equation}
Exploiting \eqref{AME:estuniU} and the isometry $\cS$ in $H^s$ ($\forall s\geq 0$), we get from \eqref{AME:PTV},
\begin{equation}
  \label{AME:estPTV}
  \|\partial_tV_{\varepsilon}\|_{L^\infty([0,T];H^{s-1})} \leq \mathcal C \, , \qquad \forall \, \eps \in ]0,\eps_0 ] \, ,
\end{equation}
where $\mathcal C$ is independent of $\varepsilon$. Using Ascoli--Arzela theorem, compact Sobolev embeddings and some interpolation inequalities, estimates \eqref{AME:estuniV} and \eqref{AME:estPTV} imply that there exists a subsequence of $V_\varepsilon$, still denoted by $V_\varepsilon$ for convenience, which converges in $\mathscr{C}([0,T];H^{\sigma})$, with $\sigma<s$, to a unique limit point denoted by $U_0=U_0(t,x)$, and such that
\[
U_0\in L^\infty([0,T]; H^s)\cap W^{1,\infty}([0,T]; H^{s-1})\,.
\]
We then introduce the two-time-scales profile $U_0^r(\tau,t,x):=\cS(\tau)U_0(t,x)$. Therefore, using the unitarity of the group $\cS$ in $H^s$ ($\forall\, s\geq 0$) and the strong convergence in $\mathscr{C}([0,T];H^{\sigma})$ of  $V_\varepsilon$ to $U_0$ as $\varepsilon$ vanishes, we obtain  for all $\sigma <s$,
\begin{eqnarray}
  \| U_\eps (t,x) - U_0^r (t/\eps,t,x) \|_{\mathscr{C}([0,T];H^{\sigma})} &=&  \| U_\eps (t,x) -\cS(t/\varepsilon) U_0(t,x) \|_{\mathscr{C}([0,T];H^{\sigma})} \nonumber\\
  &=& \| \cS(t/\varepsilon) (V_\eps (t,x) -U_0(t,x)) \|_{\mathscr{C}([0,T];H^{\sigma})}\nonumber \\
  &=& \|V_\eps (t,x) -U_0(t,x) \|_{\mathscr{C}([0,T];H^{\sigma})} \longrightarrow 0\,, \quad \mbox{ as } \ \varepsilon\rightarrow 0\,, \label{AME:SCV}
\end{eqnarray}
which is  \eqref{difflim}. The  profile $U_0^r$ is determined from $U_0$, while the time evolution of $ U_0 $ may be obtained by passing to the weak limit in the equation \eqref{AME:PTV}. Indeed, let  $\varphi \in \mathscr{C}^\infty_c([0,T[\times\TT^3)$ be a test function. Then, the weak formulation of \eqref{AME:PTV} is given by
\begin{multline}
  \label{AME:wf1}
  -\int_0^T dt \int_{\TT^3} dx\, V_\varepsilon \, \partial_t \varphi + \int_{\TT^3} dx\, V_\varepsilon^0 \,\varphi(0)
  = \int_0^T dt \int_{\TT^3} dx\,
   \varphi\, \Phi(\varepsilon,V_\varepsilon)\\ :=  \int_0^T dt \int_{\TT^3} dx\,
   \varphi(t,x)\, \cS(-t/\varepsilon)
  \big\{
  A(\varepsilon,\myunderbar{U},\cS(t/\varepsilon)V_\varepsilon(t,x), D_x)[\cS(t/\varepsilon)V_\varepsilon(t,x)]
  \\ + f(\varepsilon,\myunderbar{U},\cS(t/\varepsilon)V_\varepsilon(t,x))
  \big\} \,.
\end{multline}
Using \eqref{AME:estuniV}, we can pass to the limit in the left hand-side of \eqref{AME:wf1} and we obtain
\begin{multline}
  \label{AME:wf2}
  -\int_0^T dt \int_{\TT^3} dx\, U_0 \, \partial_t \varphi + \int_{\TT^3} dx\, U_0^0 \,\varphi(0)
  = \lim_{\varepsilon \rightarrow 0}\int_0^T dt \int_{\TT^3} dx\,
  \varphi\, \Phi(\varepsilon,V_\varepsilon)
  \\ =  \lim_{\varepsilon \rightarrow 0} \int_0^T dt \int_{\TT^3} dx\,
   \varphi(t,x) \, \cS(-t/\varepsilon) \big\{
  A(\varepsilon,\myunderbar{U},\cS(t/\varepsilon)V_\varepsilon(t,x), D_x)[\cS(t/\varepsilon)V_\varepsilon(t,x)]
  \\ + f(\varepsilon,\myunderbar{U},\cS(t/\varepsilon)V_\varepsilon(t,x))
  \big\} \,.
\end{multline}
Now, it remains to pass to the limit in the right-hand side of \eqref{AME:wf2}. Since $A$ and $f$ are smooth in their arguments, so is for $\Phi$. Since the function $\Phi$ depends smoothly of $\varepsilon$ and $V_\varepsilon$ (uniformly in $\varepsilon$), using the strong convergence \eqref{AME:SCV} of $V_\varepsilon$ to $U_0$ as $\varepsilon$ vanishes, we obtain that  $\Phi(\varepsilon,V_\varepsilon) - \Phi(0,U_0)$ converges to zero in $L^\infty([0,T];H^\sigma)$ for some $\sigma< s-1$, as $\varepsilon \rightarrow 0$. Therefore, we can replace $\Phi(\varepsilon,V_\varepsilon)$ by $\Phi(0,U_0)$ inside the limit  of \eqref{AME:wf2} to obtain
\begin{multline}
  \label{AME:wf3}
  -\int_0^T dt \int_{\TT^3} dx\, U_0 \, \partial_t \varphi + \int_{\TT^3} dx\, U_0^0 \,\varphi(0)
  = \lim_{\varepsilon \rightarrow 0} \int_{\TT^3} dx\, \int_0^T dt \, \psi(t/\varepsilon,t,x)
  \\ := \lim_{\varepsilon \rightarrow 0} \int_{\TT^3} dx
   \int_0^T dt\, \varphi(t,x)\,  \cS(-t/\varepsilon)\big\{
  A(0,\myunderbar{U},\cS(t/\varepsilon)U_0(t,x), D_x)[\cS(t/\varepsilon)U_0(t,x)]
  \\ + f(0,\myunderbar{U},\cS(t/\varepsilon)U_0(t,x))
  \big\} \,.
\end{multline}
To pass to the limit in \eqref{AME:wf3}, we use a type of Bogoliubov--Mitropolsky averaging theorem \cite{BM61, SV07} due to S. Schochet (Lemma~3.2 in \cite{Sch94a}) that we reproduce below for the sake of completeness.
\begin{lem}[Lemma~3.2 in  \cite{Sch94a}]
  \label{lem:scho}
  Let $(X, \|\cdot \|_X)$ be a Banach space. Suppose that $\psi(\tau, t, \cdot) $ is in $ L^\infty(\RR\times [0,T]; X)$, and that the following hypotheses hold:
  \begin{itemize}
  \item[(i)] For some $\upbeta\in \RR_+^*$,\
    $
      \|\psi(\tau,t, \cdot)- \psi(\tau,s, \cdot)\|_X \lesssim |t-s|^\upbeta\,.
    $ \smallskip
  \item[(ii)] The limit
    $
    \displaystyle\MM_\tau(\psi)(t,\cdot) := X\!\!-\!\!\! -\!\!\!\!\lim_{\cT \rightarrow +\infty} \frac{1}{\cT}\int_0^\cT d\tau\,\psi(\tau,t,\cdot)
    $
    exists. \smallskip
  \item[(iii)] $\displaystyle\lim_{\cT \rightarrow +\infty}\delta(\cT)=0$, with
    \[
    \delta(\cT):=\sup_{\cT'>\cT} \ \sup_{t\in[0,T]} \ \sup_{\upalpha \in \RR} \bigg \|\, \MM_\tau(\psi)(t,\cdot) - \frac{1}{\cT'}\int_{\upalpha}^{\upalpha+\cT'}\!\! d\tau \, \psi(\tau,t,\cdot) \bigg  \|_X\,.
    \]
  \end{itemize}
  Then $\MM_\tau(\psi) \in \mathscr{C}([0,T]; X)$,
  \[
  X\!\!-\!\lim_{\varepsilon\rightarrow 0} \int_0^t ds \, \psi(s/\varepsilon, s, \cdot) = \int_0^t ds \, \MM_\tau(\psi)( s, \cdot)\,,
  \]
  and this limit converges uniformly for $t\in[0,T]$.
\end{lem}
If the function $\psi$, defined by \eqref{AME:wf3}, satisfies all the hypotheses of Lemma~\ref{lem:scho} with $\upbeta=1$ and $X=H^{s-1}$, then applying  Lemma~\ref{lem:scho} to \eqref{AME:wf2} we obtain
\begin{multline*}
  \label{AME:wf4}
  -\int_0^T dt \int_{\TT^3} dx\, U_0 \, \partial_t \varphi + \int_{\TT^3} dx\, U_0^0 \,\varphi(0)
  \\ =  \int_0^T dt  \int_{\TT^3} dx
  \, \varphi(t,x) \, \MM_\tau \Big[ \cS(-\tau)\big\{
  A(0,\myunderbar{U},\cS(\tau)U_0(t,x), D_x)[\cS(\tau)U_0(t,x)]
  \\ + f(0,\myunderbar{U},\cS(\tau)U_0(t,x))
  \big\} \Big]\,,
\end{multline*}
which is the weak formulation of the modulation equations \eqref{modu0}-\eqref{modu0ini0}. In fact, from the conclusion of  Lemma~\ref{lem:scho} and the regularity of $U_0$, this limit,  i.e.  the modulation equations \eqref{modu0}-\eqref{modu0ini0}, holds in $L^\infty([0,T];H^{s-1})$ with $s>5/2$. Therefore, it remains to verify the hypotheses of Lemma~\ref{lem:scho}, with $\upbeta=1$ and $X=H^{s-1}$ for $s>5/2$.

Since $A$ and $f$ are smooth in their arguments, and because
\[
U_0\in L^\infty([0,T]; H^s)\cap W^{1,\infty}([0,T]; H^{s-1}) \, , \qquad s>5/2 \, ,
\]
the continuity of the map $\cS$ in $H^r$ with $r\geq 0$, and the algebra property of $H^r$ for $r>3/2$, lead to  $(i)$ with  $\upbeta=1$ and $X=H^{s-1}$. The remaining hypotheses $(ii)$--$(iii)$ of Lemma~\ref{lem:scho} are  typical properties of Almost-Periodic functions (AP) functions. Indeed, for any Banach space $Y$, the set ${\rm AP}(\RR_\tau; Y)$ of $Y$-valued almost-periodic functions is the set of functions $h\in \mathscr{C}(\RR_\tau; Y)$ for which the following defining property of AP functions holds \cite{Cor89}.

{\it A property of AP functions.} For any number $\eta >0$, there exists  a number $\ell(\eta)>0$ such that any interval on $\RR_\tau$ of length $\ell(\eta)$ contains at least one point $p$ such that
\[ \| h(\tau+ p,\cdot)-h(\tau,\cdot)\|_Y \leq \eta \, , \qquad \forall \, \tau\in\RR \, . \]

The linear properties of ordinary $\RR$-valued AP functions remain valid for AP functions with values in Banach spaces \cite{Cor89}. In particular, the mean value
$
Y\!\!-\lim_{\cT \rightarrow +\infty} \frac{1}{\cT}\int_{\upalpha}^{\upalpha+\cT} d\tau\,h(\tau,\cdot)
$
of almost-periodic functions exists and is independent of $\upalpha$, and the convergence is uniform in $\upalpha \in \RR$. Therefore, properties $(ii)$--$(iii)$ of Lemma~\ref{lem:scho} hold if $\psi \in {\rm AP}(\RR_\tau; L^\infty([0,T];X))$; this is what remains to be shown with $X=H^{s-1}\hookrightarrow L^\infty$ for $s>5/2$.

Since the restriction of $\cS $ to every Fourier mode is periodic, starting from the information  $U_0 \in L^\infty([0,T]; H^s)\cap W^{1,\infty}([0,T]; H^{s-1})$, we obtain
\[
\cS(\tau)U_0 \in {\rm AP} \bigl(\RR_\tau; L^\infty([0,T]; H^s)\cap \mathscr{C}([0,T];H^\sigma) \bigr) \, , \qquad  \forall \sigma<s \, .
\]
It follows that
\[
D_x\cS(\tau)U_0 \in {\rm AP}\bigl(\RR_\tau; L^\infty([0,T]; H^{s-1}) \bigr) \, .
\]
Since $A_j$ and $f$ are smooth functions  in their arguments, their compositions with $\cS(\tau)U_0 $ belong to the space ${\rm AP}(\RR_\tau; L^\infty([0,T]; H^{s}))$. From the stability of the multiplication in ${\rm AP}(\RR_\tau; L^\infty([0,T]; H^{s-1}))$ for all $s>5/2$, we can assert that $ A(0,\myunderbar{U},\cS(\tau)U_0, D_x)\cS(\tau)U_0 $ is in $ {\rm AP}(\RR_\tau; L^\infty([0,T]; H^{s-1}))$.

Finally, since $\cS(\tau) h \in {\rm AP}(\RR_\tau; L^\infty([0,T]; H^{s-1}))$ for all $ h \in {\rm AP}(\RR_\tau; L^\infty([0,T]; H^{s-1}))$ as well as the stability of the addition in ${\rm AP}(\RR_\tau; L^\infty([0,T]; H^{s-1}))$, we obtain that $\psi $ is in $ {\rm AP}(\RR_\tau; L^\infty([0,T];H^{s-1}))$, which completes the verification of the hypotheses in  Lemma~\ref{lem:scho}.


\section{The origins  of XMHD}
\label{s:appendix}
The study of XMHD lies at the interface between physics and mathematics. In the domain of space and laboratory plasmas, it is viewed as a major contemporary issue \cite{SCD09}, linked for instance to X-point collapse \cite{AL} or to the modeling of solar wind \cite{ALM}. Nevertheless, in the mathematical community, this theme is being at this point largely unknown.

The aim of this appendix is to build a bridge concerning XMHD between these two scientific disciplines. On the one hand, we would like to explain why our derivation of XMHD comes as a valuable complement to previous efforts \cite{AKY,DML16,GP04} to lay the foundation for XMHD theory. On the other hand, we would like to make the XMHD topic more accessible and more attractive to mathematicians. There are indeed several challenging  open problems in this direction.

In this appendix, we better identify the role of XMHD (as compared to  MHD); we  recall the different ways to reach XMHD; we  clarify the advantages\footnote{Our method of derivation of XMHD from EMTF applies within a more general framework than before. To a certain extent, it is also more intrinsic and more informative.} arising from our asymptotic analysis; and we get a better grasp of a whole range of phenomena often regarded as falling within ``plasma turbulence''.

The expression ``plasma turbulence'' is used in this text to designate movement patterns (including the Hall and inertial effects as well as fast magnetic reconnection) arising at inertial scales, and standing outside MHD rules.

First and foremost, we have to identify the scope of XMHD application, which stems from a specialization of the various regimes contained in EMTF. The EMTF equations depend on numerous plasma parameters \cite{GIP}. Written in dimensionless variables, they reveal the significance of the normalized electron and ion skin depths $ d_e $ and $ d_i $ given by \eqref{eqn:cteXMHD}, with in general $ 0 < d_e \lesssim d_i \lesssim 1 $. In this dimensionless representation:
\begin{itemize}
 \item [$-$] MHD may be a reasonable approximation as long as the scale $ \ell $ of observation is of size one, that is $ \ell \sim 1 \, $;
 \smallskip
 \item [$-$] Hall MHD entails that $ \ell \sim d_i \, $;
 \smallskip
 \item [$-$] XMHD is essential in the inertial regime, that is when $ \ell \sim d_e $, that is where the distinction between the inertia of the two species is most obvious\footnote{In other words, assuming that $ d_e = d_i = 0 $, XMHD coincides with MHD. For $ 0 < d_i \ll 1 $ and low frequencies, Hall MHD can be viewed as a slight modification of MHD. Moreover, if $ 0 < d_e \ll d_i $, the same applies to XMHD with respect to Hall MHD (and therefore MHD). Of course, at high frequencies $ \sim 1/d_i $ or $ \sim 1/d_e $, the Hall and inertial effects become predominant  (and MHD  is no longer relevant).}.
\end{itemize}

   At scales $ \ell \sim d_e $, which are in an intermediary zone between the macroscopic perception of ideal MHD and the mesoscopic precision of kinetic theory\footnote{XMHD does not take into account kinetic phenomena (such as Landau damping or whistler waves) or dissipative effects, which are supposed to occur for $ \ell \ll d_e $.}, the understanding of plasma dynamics requires specific tools and adapted descriptions \cite{SCD09}. The XMHD model is quite suitable for that purpose because it captures most two-fluid impacts arising in this context, leading to numerous interesting applications.

XMHD has been discovered in 1959 by V.R. L\"ust \cite{Lu} through formal computations based on the EMTF system. The usual two-fluid approximations  leading to XMHD require the neutrality ($ {\rm n}_e = Z \, {\rm n}_i $), while neglecting some contributions with $ \updelta := Z m_e/m_i $ in factor. More precisely, Hall and XMHD have been  obtained by retaining respectively zero-order and first-order terms in powers of $ \updelta $.

Such  simplifications are not at all required in our text. Instead, they are replaced by the smallness of the  perturbation $ (\eps \ll 1 $). As stated in Theorem \ref{maintheo}, it suffices to retain \eqref{neutral} and to implement prepared data in order to recover XMHD as a limit model. In other words, the Wentzel--Kramers--Brillouin  hierachy (in powers of $ \eps $) offers an alternative to justify the XMHD model  (without implementing artificial assumptions), as a  universal behavior that  appears after a certain period of time ($ \tau \sim 1/ \eps $ or $ t \sim 1 $). Moreover, it sheds new light on XMHD because it provides the SLM reformulation of XMHD where it is possible to keep track of the 14  components (of $ {\rm U} $) and to maintain the symmetric structure of the original quasilinear EMTF system.

In Subsection \ref{sub:fluid}, we  sketch out the standard derivation of XMHD from EMTF. This facilitates a better understanding of the reduction process\footnote{One option is to suppress by {\it ad hoc} assumptions well chosen terms inside EMTF. It is more complicated to show that such contributions disappear from the time evolution inside lower order terms  when $ \eps \rightarrow 0 $.} yielding XMHD. At the same time, this helps identify the physical meaning of the terms which are incorporated when passing from the  MHD description to the XMHD one.

The Hamiltonian formalism of XMHD has recently gained a growing interest \cite{AKY,DML16,KLMLW,KM}. In Subsection \ref{sub:hamiltonian}, we briefly describe this interesting direction of research and its important outcomes \cite{AL,ALM,MLM,SMA}.

XMHD can also be addressed from the perspective of mathematics. However, while Hall effects have already been investigated \cite{DLT,JO} in this way, the analysis of inertial aspects has been little explored so far. In Subsection \ref{sub:math}, we mention the principal obstacles explaining this situation. The preceding facets of XMHD, fluid in Subsection \ref{sub:fluid}, Hamiltonian in Subsection \ref{sub:hamiltonian}, and related to  mathematics in  Subsection \ref{sub:math}, provide an entry point for the study of ``plasma turbulence''. As a matter of fact, in the  inertial regime, the fluid motions share  characteristics, denoted below C1), C2) and C3), with turbulent flows. Let us explain rapidly (details will be  provided in the subsections) how these features occur concomitantly in the XMHD context.

\begin{itemize}
\item[C1)] {\it Fluctuations around a mean flow} (Subsection \ref{sub:fluid}). In our context, the ``average'' flow is induced by the center of mass velocity\footnote{The field $ {\rm u} $ takes here the place of the usual  time averaging, spatial averaging or ensemble averaging. It indicates the mean behavior of the two-fluid flow, without going into the  details carried by $ {\rm v}_e $ and $ {\rm v}_i $, without looking at the ``fluctuating parts'' of the velocities, which are $ {\rm v}_i - {\rm u} $ and $ {\rm v}_e - {\rm u} $.}
\begin{equation}
  \label{centerofmass}
        {\rm u} := (m_e \, {\rm n}_e \, {\rm v}_e + m_i \, {\rm n}_i \, {\rm v}_i)/ (m_e\, {\rm n}_e + m_i \, {\rm n}_i) \,.
\end{equation}
The equation on $ {\rm u} $ reveals the presence of a two-fluid Reynolds stress $ \mathrm R $,  defined at the level of \eqref{reytensor}. Through quasineutrality and then through the  magnetostatic assumption (in physics) or the prepared condition (in mathematics), the XMHD approach furnishes an interpretation of $ \mathrm R $ first in terms of the total current density and secondly in terms of $ \nabla \times B $. This is a sort of electromagnetic version of the two-fluid velocity tensor $ \mathrm R $. From this derives the contribution\footnote{The constitutive relation \eqref{lienBB*2} furnishes $ B^* \times (\nabla \times B) = B \times (\nabla \times B) - (d_e^2/\myunderbar{\rho})\Delta B \times (\nabla \times B)$. The first part, namely  $ B \times (\nabla \times B) $,  is common in ideal MHD \cite{BC2,JJLX,JSX19}. But the second is not. It comes from $ {\rm R} $. It takes into account the underlying presence of two distinct species.} $ B^* \times (\nabla \times B) $ in the first equation of  \eqref{syssimpli}.
In addition, resorting to the generalized Ohm's law allows to eliminate the electric field $ E $, and to extract an equation\footnote{Compare the second equation of \eqref{syssimpli} with those (on $ B $) inside  \cite{BC2,JJLX,JSX19} to realise that many terms (which reflect the richness of two-fluid effects) have been added.} on $ B^* $. In \cite{BC1}, it is proved that the well-posedness of \eqref{syssimpli} can be recovered by using the state variables $ \nabla \times u $ and $ B^* $, which lead to hyperbolicity. The fact that XMHD gives a central role to the vorticity  is already indicative of its  deep relation to turbulence. \smallskip
\item[C2)] {\it Statistical physics and power laws} (Subsection \ref{sub:hamiltonian}). The Hamiltonian formalism  \cite{AKY,DML16, KMo} is an alternative way to approach XMHD. By this way, statistical theory finds application in XMHD with specific cascades \cite{MLM} and energy spectra involving unusual power-law indexes \cite{ALM}.\smallskip
\item[C3)] {\it Instabilities} (Subsection \ref{sub:math}). In general, we find $ d_e \ll d_i $. If $ d_ e $ is neglected, the  destabilizing Hall effects \cite{ChaeW,JO} are predominant (at least if the focus is on lengths of size $ d_i $), which  might give the impression of uncontrolled instabilities (leading to ill-posedness)\footnote{The existence of strong amplification mechanisms due to the Hall contribution is fairly recent \cite{ChaeW,JO}}. However, beyond this point, the inertial effects come to  restore the stability \cite{BC1}. The crucial thing is that the Hall term\footnote{The Hall term appears when replacing $ u $ by $ u - (d_i/\myunderbar{\rho})\nabla \times B $ in the second line of \eqref{syssimpli}.} is often added and isolated into MHD equations, while it should be considered  as a  contribution among others (in competition with other effects) inside the two-fluid equations. In the inertial range, the hyperbolic structure is not inherited from MHD or Hall MHD. Instead, it comes from EMTF through SLM. Thus, the description of wave propagation in terms of magnetosonic and Alfv\'en waves is no more  adapted.
Other dispersion laws come into play\footnote{The standard MHD dispersion relations are just approximations of those of XMHD. While they provide meaningful predictions at low frequencies, they strongly differ from reality at high frequencies.} \cite{BC1,SMA}, and they become fundamental at high frequencies to understand how information can propagate.

Now, the transition between the ideal and inertial regimes can be brutal near some positions, with an impression of unpredictability. Still, the explanation can remain deterministic by changing MHD into XMHD.
\end{itemize}

To our knowledge, the link between the incompressible (or the low Mach number limit) of EMTF and  XMHD had never been established before. It is therefore interesting to revisit the preceding approaches in light of our alternative viewpoint. This is what we do in the next subsections.


\subsection{XMHD viewed from two-fluid mechanics}
\label{sub:fluid}
In this subsection, we present the  various approximations which led physicists to consider XMHD, and we put them into perspective with our less restrictive framework. To simplify the presentation, from now on, we work with $ Z = 1 $.
The starting point is the system \eqref{TFEM}. Summing the mass  densities, we get the total mass density $ \, \rho := \rho_e + \rho_i \, $ where $ \, \rho_s := m_s \, {\rm n}_s$. Summing the momenta, we can define the center of mass velocity $ {\rm u} $ as indicated in \eqref{centerofmass}. By construction, these quantities must satisfy the net continuity equation
\begin{equation}
  \label{netcon}
  \partial_t \rho \ \, + \, \nabla \cdot (\rho \, {\rm u} ) \ \,  = 0 \,.
\end{equation}
Summing charge densities, current densities and pressure tensors, we obtain
\[
\begin{array}{lll}
  $--\ $\text{the total charge density:} &
  \sigma := \sigma_e + \sigma_i \,, \quad & \sigma_e := -\,{\rm n}_e \,, \ \ \sigma_i := {\rm n}_i\,, \\
  $--\ $\text{the total current density:} &
         {\rm J} := {\rm J}_e + {\rm J}_i \,, &  {\rm J}_s := \sigma_s \, {\rm v}_s \,, \\
  $--\ $\text{the total pressure tensor:} &
    {\rm P} := {\rm P}_e + {\rm P}_i \,, & {\rm P}_s := {\rm p}_s ({\rm n}_s) \, \id \,.
\end{array}
\]
Summing the momentum  equations of both species, we find the net momentum equation
\begin{equation}
  \label{netmomentum}
  \partial_t (\rho \, {\rm u} ) + \nabla \cdot (\rho \ {\rm u} \otimes {\rm u} + {\rm R} + {\rm P} ) = \sigma \, {\rm E} + {\rm J} \times {\rm B} \,,
\end{equation}
where $ {\rm R} $ is the ``Reynolds stress'' (or the ``two-fluid diffusion velocity tensor'') given by
\begin{equation}
  \label{reytensor}
  \qquad \quad {\rm R} := \sum_{s \in \{e,i\}} \! \!\! \rho_s \ {\rm v}_s \otimes {\rm v}_s - \rho \ {\rm u} \otimes {\rm u} = \rho_e \, {\rm w}_e \otimes {\rm w}_e + \rho_i \, {\rm w}_i \otimes {\rm w}_i , \qquad {\rm w}_s := {\rm v}_s- {\rm u} \,.
\end{equation}
The system \eqref{netcon}-\eqref{netmomentum} is not closed, due to the presence of $ {\rm R} $.
Some ``anomalous'' diffusion  is  sometimes  proposed \cite{SCD09} to model the impact of $ {\rm R} $. Another    alternative, which is more in the spirit of the pioneering work of L\"ust \cite{Lu}, is to impose special assumptions, as decribed below in Paragraphs \ref{subsub:neutrality}, \ref{subsub:ohm} and \ref{subsub:magnetostatic}.


\subsubsection{The neutrality assumption}
\label{subsub:neutrality}
The first approximation usually made to extract XMHD is to impose a complete neutrality, that is
\begin{equation}
  \label{completeneutrality}
 \sigma  = {\rm n}_i-{\rm n}_e = 0 \,.
\end{equation}
By contrast, in our setting, the fluid is quasineutral (this hypothesis holds on distances larger than the Debye length). This is due to \eqref{neutral} which furnishes
\begin{equation}
  \label{quasine}
 \sigma_\eps = \eps \,( n_{i\eps} - n_{e\eps}) = \mathcal{O}(\eps) \,.
\end{equation}
Observe that
\[
\int_{\TT^3} (n_{i\eps} - n_{e\eps}) (t,x) \, dx = \int_{\TT^3} (n_{i0}^0 - n_{e0}^0) (x) \, dx +o(1) \,.
\]
In the unprepared case (when $ n_{i0}^0 $ may be distinct from $  n_{e0}^0 $), the information \eqref{quasine} cannot be improved into some $ o(\eps) $. It remains true, however, that the total charge density $ \sigma_\eps $ vanishes rapidly when $ \eps $ goes to zero: the condition \eqref{completeneutrality} is always recovered approximately in the limit process.
Assuming \eqref{completeneutrality}, we can define $ {\rm n} := {\rm n}_e = {\rm n}_i $. Then, from the relations $ {\rm J} =  {\rm n} \, ({\rm w}_i-{\rm w}_e) $ and $ m_i \, {\rm w}_i + m_e \, {\rm w}_e = 0 $, we can express the relative velocities $ {\rm w}_s $ and $ {\rm R} $ in terms of $ {\rm J} $ according to
\[
  {\rm w}_i = m_e \,  \frac{{\rm J}}{\rho}, \qquad {\rm w}_e = m_i \, \frac{\rm J}{\rho} , \qquad {\rm R} =  m_e \, m_i \, \frac{{\rm J} \otimes {\rm J}}{ \rho} \,.
\]
Taking all these into account, there remains
\begin{equation}
  \label{onyva}
  \partial_t (\rho \, {\rm u} ) + \nabla \cdot \Bigl(\rho \ {\rm u} \otimes {\rm u} + m_e \, m_i \,
  \frac{{\rm J} \otimes {\rm J}}{ \rho} + {\rm P} \Bigr) =  {\rm J} \times {\rm B} \,.
\end{equation}
Since the total current density $ {\rm J} $ does not depend only on $ (\rho,{\rm u},{\rm E},{\rm B}) $, the preceding closure problem remains unsolved. Furthermore, some information is lost when replacing the two first equations of \eqref{TFEM} by their sum. It is necessary to also consider the difference of these two  equations, as in the next paragraph.


\subsubsection{The Generalized Ohm's Law (GOL)}
\label{subsub:ohm}
We refer to \cite{KMo}, and especially to Table I in \cite{KMo},  for a discussion about GOL. By subtracting the two momentum equations of \eqref{TFEM} and using \eqref{eqn:cteXMHD}, we obtain
\begin{equation}
  \label{Ohmlaw}
 \begin{array}{rl}
    {\rm E} \,+\, {\rm u} \times {\rm B} \! \! \!  & \displaystyle = {- \, \frac{d_i}{\rho} \ \nabla {\rm p} \, +\,  d_i \ \frac{\rm J}{\rho} \times {\rm B}} \  {- \ d_i \, d_e^2 \ \bigg( \frac{\rm J}{\rho} \cdot \nabla \bigg) \frac{\rm J}{\rho}} \medskip \\
    \ & \displaystyle \quad {+ \ d_e^2 \, \bigg \lbrack \partial_t \bigg( \frac{\rm J}{\rho} \bigg) + ({\rm u} \cdot \nabla)  \bigg( \frac{\rm J}{\rho} \bigg) +  \bigg( \frac{\rm J}{\rho} \cdot \nabla\bigg) {\rm u}
   \bigg \rbrack} \,.
   \end{array}
\end{equation}
The relation \eqref{Ohmlaw} is another way to reveal the role of the two independent dimensionless  parameters ${d_i}$ (for the {\underline{i}}on {\it skin depth}) and ${d_e} $ (for the {\underline{e}}lectron {\it skin depth}), with in general $ 0 < d_e \lesssim d_i \lesssim 1 $. Then, we can again identify three different regimes. Ideal MHD occurring when $ 0 ={d_e} = {d_i} $; Hall MHD when $ 0 ={d_e} < {d_i} $; and  Inertial MHD when $ 0 < {d_e} \lesssim {d_i} $, which is the most complete regime. For that reason, XMHD (instead of MHD) should be considered as the correct paradigm for a simplified description of plasma dynamics.

The formula \eqref{Ohmlaw} lays emphasis on the global Lorentz force. Ideal MHD is pertinent when this force is zero; Hall MHD occurs when it undergoes a slight modification (with the inclusion of the first and second terms of the right-hand side of \eqref{Ohmlaw}); and XMHD takes over when it is complemented by all remaining terms (i.e. the third and fourth terms of the right-hand side of \eqref{Ohmlaw}).
The presentation \eqref{Ohmlaw} gives little weight to the implementation in the right-hand side of a time derivative, namely $ d_e^2 \, \partial_t ( {\rm J}/\rho ) $ with $ d_e^2 \ll 1 $. In coherence with \eqref{TFEM} and more in line with our approach, another option is to restore the dynamical viewpoint, that is to look at \eqref{Ohmlaw} as a penalized evolution equation on $ {\rm J}/\rho $.

Seen from this perspective, the MHD and Hall MHD situations correspond to stable manifolds of the dynamical system. And the flow responds to the departure from MHD or Hall equilibria by creating very rapid time variations,  i.e. time derivatives of size $ d_e^{-2} \gg 1 $. Note that this phenomenon is partly damped when starting from prepared data.

The GOL is an essential tool in the usual derivations \cite{AKY,DML16,KLMLW,KMo,LTHG,Lu} of XMHD. Thus, one might wonder why the GOL does not appear in our proof ? In fact, it does play an important role in the background for reasons that are detailed below:
\begin{itemize}
\item An alternative approach to the filtering unitary group method is to perform a more complet WKB calculus. By expanding the electromagnetic field $ ({\rm E}, {\rm B}) $ into a power series (in powers of $ \eps $), say $ \eps \, (E_0,B_0) + \eps^2 \, (E_1,B_1) + \cdots $, a substitute for \eqref{Ohmlaw} can be found. Roughly speaking, this amounts to replace $ ({\rm E}, {\rm B}) $ inside \eqref{Ohmlaw} by the first corrector $ (E_1,B_1) $.
\item In our asymptotic analysis, the impact of the GOL is hidden behind the action of the projector $ \PP $, which combines the different equations properly in order to produce the right hand side of \eqref{Ohmlaw}, while the left hand side is eliminated in the case of prepared data.
\end{itemize}

The next step is to express $ {\rm J} $ in terms of $ {\rm B} $, and to remove $ {\rm E} $ by exploiting the GOL.


\subsubsection{The magnetostatic assumption}
\label{subsub:magnetostatic}
This is when the displacement current is negligible. Taking into account  Amp\`ere--Maxwell law, this means that
\begin{equation}
  \label{magnetos}
  \nabla \times {\rm B} - {\rm J} = 0\,. 
\end{equation}
By contrast, in our framework, the last equation of \eqref{TFEM} gives rise to
\[
\partial_t E_\eps + \, J_\eps = \frac{1}{\eps} \
\Big( \nabla \times B_\eps - \underline{n} \, (v_{i\eps}-v_{e\eps}) \Big) , \qquad J_\eps := n_{i \eps} \, v_{i\eps} -  n_{e \eps} \, v_{e\eps}  \,.
\]
The singular term (with $ 1/\eps $ in factor) disappears in the case of prepared data, when
\begin{equation}
  \label{premagn}
  \nabla \times B_0^0 - \, \underline{n}\, (v_{i0}^0-v_{e0}^0) = 0 \,.
 \end{equation}
In other words, under \eqref{premagn}, the equation \eqref{magnetos} is satisfied modulo non trivial error terms. It is evident that the condition \eqref{magnetos} is much more restrictive than \eqref{premagn}. Now, assuming \eqref{magnetos}, the field $ {\rm J} $ becomes a function of $ \nabla \times {\rm B} $.  It is a clear advantage. From there, the identity \eqref{Ohmlaw} can be reformulated into
\[
  {\rm E} = d_e^2 \ \partial_t \big((\nabla \times {\rm B}) / \rho\big) + \cE (\rho,{\rm u},{\rm B}) \,,
\]
where $ \cE (\rho,{\rm u},{\rm B})$ stands for a nonlinear partial differential operator on $ (\rho,{\rm u},{\rm B}) $, involving only spatial derivatives. After substitution in the Maxwell--Faraday equation, there remains
\begin{equation}
  \label{onyvaa}
  \partial_t \big (\underbrace{ {\rm B} + {d_e^2 \ \nabla \times ( \nabla \times {\rm B} / \rho )}}_{= {\rm B}^*}  \big) + \nabla \times \mathcal E(\rho,{\rm u},{\rm B}) = 0 \,.
\end{equation}
This explains the origin of the ``dynamical'' state variable $ {\rm B}^* $, sometimes called ``drifted'' magnetic field, which is involved everywhere in the articles  \cite{AKY,DML16,KLMLW,KMo,LTHG,Lu}.

Combining \eqref{onyva} and \eqref{onyvaa}, we obtain a closed system on $ (\rho,{\rm u},{\rm B^*}) $, called the compressible XMHD equations \cite{BC1}. The  incompressible version \eqref{syssimpli} can be obtained by  assuming that $ \rho $ is constant and by introducing a Lagrange multiplier (in place of the pressure law).
From a theoretical point of view, the introduction of $ {\rm B^*} $ is not at all necessary. Besides, it is not a prerequisite at the level of SLM\footnote{The role of Theorem \ref{maintheo} is to make the transition from $ ({\rm u},{\rm B}^*) $ to $ U_0 = \PP \, U_0$ (from which $ {\rm B}^* $ is absent).}. And it is a little bit misleading because it leads to a presentation of XMHD which connects the two state variables $ {\rm u}$ and $ {\rm B^*} $. Yet, in order to recover well-posedness, we know from \cite{BC1} that $ \nabla \times {\rm u}$ and $ {\rm B^*} $ (or $ {\rm u} $ and the magnetic potential $ {\rm A^*} $ with $ \nabla \times {\rm A^*} = {\rm B^*}$) should be associated.
Another option is to consider directly the redressed state variable $ \cU_0 $.

The formulation \eqref{syssimpli} of XMHD has the clear advantage of simplicity. It avoids the need to use Fourier multipliers, as in \eqref{modu0prepa} through $ \PP $. But it is strongly restricted to the manipulation of prepared data.
One benefit of our asymptotic analysis is to justify the FLM model \eqref{modu0} in the case of general data. The equation \eqref{modu0} encompasses both \eqref{modu0prepa} and \eqref{syssimpli}. It provides a natural extension to the XMHD model \eqref{syssimpli} of physicists, by  taking into account the impact of oscillations. The system \eqref{modu0} is not easily accessible but it is more intrinsic and more general in scope. It would be interesting to interpret in terms of physics the several new inputs brought by \eqref{modu0}.


\subsection{XMHD viewed from Hamiltonian  mechanics}
\label{sub:hamiltonian}
XMHD can be interpreted as a Hamiltonian system like $ \part_t \cV = \cJ(\cV) \ \part_{\cV} \cH (\cV) $, where the state vector $ \cV $ stands for a point in a Hilbert space. The content of the Poisson operator $ \cJ $ and of the Hamiltonian $ \cH $ is detailed in \cite{AKY}. In the compressible version\footnote{See \cite{BC1} for the incompressible case.} of XMHD, the infinite dimensional phase space is spanned by $ (\rho,u,B^*) $ and the conserved energy $ \cH $ is given (for some internal energy function $ U $) by
\[
\cH= \int_{\TT^3} \left \lbrace \rho \, \bigg( \frac{\vert u\vert^2}{2} + U(\rho) \bigg)
+ \frac{\vert B\vert^2}{2} + d_e^2 \ \frac{\vert \nabla \times B\vert^2}{2 \, \rho} \right \rbrace \, dx , \qquad B^*= B + d_e^2 \ \nabla \times \bigg( \frac{ \nabla \times B}{\rho} \bigg) \,.
\]
The Hamiltonian formalism of XMHD is highly instructive. It has been recently developed in \cite{AKY,DML16,KLMLW,KMo,LMMbis}, see also the references therein. A noncanonical Poisson bracket can be specified \cite{AKY,DML16}. There is a non-trivial center for  the governing Poisson algebra, yielding Casimir elements (helicities) \cite{AKY,LMMbis}. This approach helps to elucidate geometrical structures inside XMHD. It leads also to interesting applications.

The first outcome is about  magnetic reconnection, which happens when the topology of magnetic lines may change under the evolution. While resistive reconnection is today a well-established fact \cite{CCL}, the understanding of the  collisionless  reconnection mechanisms\footnote{Collisionless (also called ``fast'' or ``inertial'') reconnection is ubiquitous in laboratory and astrophysical plasmas. Prototypes of such phenomena are for instance given by solar flares.} remains in progress  \cite{AL,AGMDG,GTAM}. Mathematically, this raises the interesting open question of singularity formation for  XMHD.

Another notable motivation is the study of turbulence through the tools of statistical theory\footnote{As we have done, this means to consider periodic solutions and to perform a Fourier analysis. But that is where the comparison ends.}  \cite{AGMDG,SCD09}. In the inertial regime, there are power laws with specific spectral indices \cite{ALM}. For XMHD, there are also a Liouville's theorem, absolute equilibrium states and (direct) inertial cascades \cite{MLM}.


\subsection{XMHD viewed from mathematical perspectives}
\label{sub:math}
XMHD is the continuation of recent advances in the domain of partial differential equations. It can be inferred either as an extension of MHD (Paragraph \ref{subsub:math1}) or from EMTF (Paragraph \ref{subsub:math2}).


\subsubsection{From MHD to XMHD}
\label{subsub:math1}
The reader can refer to the introduction of the article \cite{BC1} for historical details about the passage from MHD to Hall MHD, and ultimately to XMHD. In fact, progress has been made step by step. In past years, the focus was mainly on Hall MHD with stability results in the presence of resistivity \cite{Chae} and strong amplification mechanisms in the purely  collisionless context \cite{ChaeW,JO}. It is only relatively lately  \cite{BC1} that inertial effects ($ d_e > 0 $) have been incorporated. At first sight, the equations of XMHD appear much more complicated than those of Hall MHD, and they may give a false impression of  ill-posed. This is probably why XMHD has been little explored.


\subsubsection{From EMTF to XMHD}
\label{subsub:math2}
The importance of EMTF for plasma physics has long been recognized. The EMTF system falls under the scope of quasilinear symmetric systems. As such, it is locally well-posed. From there, most efforts have focused on the long time existence, with the help of dissipation or relaxation terms  \cite{DLZ,ISYG,XXK} or under smallness and irrotational assumptions \cite{GIP}.

There are also many asymptotic results concerning EMTF, in various regimes implying Raman instabilities \cite{DLT} or non-relativistic and quasineutral features, see  \cite{LPX,LPW,PL,PXZ} and the numerous references therein. The common point of these latter contributions is to work under very restrictive assumptions, which are aimed at eliminating most two-fluid effects. As we have explained in this text, a key to go further is to take into account the full range of inertial effects.

Besides, the low Mach number limits  have been intensively studied for compressible Euler systems \cite{Ala08,Ga,KM,K,LM,MS,S07} and even for ideal MHD equations \cite{BC2,JJLX,JSX19}. But, curiously enough in view of the fundamental importance of EMTF, they do not appear to have been considered in a two-fluid context. Again, the major obstacle comes from a good understanding of  inertial features.

\medskip

{\bf Conclusions}. Theorem \ref{maintheo} makes for the first time the link between
the important theory of incompressible limits  (or of low Mach number limits) and XMHD. This confirms that XMHD is a fundamental model in plasma physics, which is more relevant and complete than MHD in many respects.

Indeed, XMHD simplifies the description of two-fluid conductive media. It preserves many essential two-fluid features and it furnishes  indispensable means of absorbing Hall instabilities and beyond of investigating inertial physics. XMHD should draw the attention in the coming decades. It is a gateway to many contemporary scientific challenges such as plasma turbulence, inertial effects  and fast magnetic reconnection. It is a wealth of interesting problems at the interface between physics and mathematics.

\bigskip


\begin{thebibliography}{9}
{\small
  
\bibitem{AKY}
  \newblock H.-M.~Abdelhamid, Y.~Kawazura, Z.~Yoshida,
  \newblock {\it Hamiltonian formalism of extended magnetohydrodynamics},
  \newblock J. Phys. A: Math. Theor. {\bf 48} (2015) 235502.
  
\bibitem{AL}
  \newblock H.-M.~Abdelhamid, M.~Lingam,
  \newblock {\it Hamiltonian formulation of X-point collapse in an extended  magnetohydrodynamics framework},
  \newblock Phys. Plasmas {\bf 31} (2024) 102104.

\bibitem{ALM}
  \newblock H.-M.~Abdelhamid, M.~Lingam, S. M.~Mahajan,
  \newblock {\it Extended MHD turbulence and its applications to the solar wind},
  \newblock Astrophys. J. {\bf 829} (2016) 87.

\bibitem{Ala08}
  \newblock  T.~Alazard,
  \newblock {\it A minicourse on the low Mach number limit},
  \newblock  Discrete Contin. Dyn. Syst. Ser. S {\bf 1} (2008) 365--404.

\bibitem{AGMDG}
  \newblock N.~Andr\'es, C.~Gonzalez, L.~Martin, P.~Dmitruk, D.~G\'omez,
  \newblock {\it Two-fluid turbulence including electron inertia},
  \newblock Phys. Plasmas {\bf 21} (2014) 122305.

\bibitem{BC1}
  \newblock  N.~Besse, C.~Cheverry,
  \newblock {\it The equations of extended magnetohydrodynamics},
  \newblock SIAM J. Math. Anal. {\bf 57} (2025) 4519--4555.

\bibitem{BC2}
  \newblock N.~Besse, C.~Cheverry,
  \newblock {\it Singular limits of anisotropic weak solutions to compressible magnetohydrodynamics},
  \newblock arXiv:2506.19784 (2025).

\bibitem{BM61}
  \newblock N.N~Bogoliubov, Y.A.~Mitropolsky,
  \newblock {\it Asymptotic methods in the theory of non-linear oscillations},
  \newblock Gordon and Breach, 1961.
  
\bibitem{CCL}
  \newblock P.~Caro, G.~Ciampa, R.~Luc\`a,
  \newblock {\it Magnetic reconnection in magnetohydrodynamics},
  \newblock Rev. Mat. Iberoam. {\bf 41} (2025) 461--508.

\bibitem{Chae}
  \newblock D.~Chae, P.~Degond, J.G.~Liu,
  \newblock {\it Well-posedness for Hall-magnetohydrodynamics},
  \newblock Ann. Inst. H. Poincar\'e Anal. Non Lin\'eaire {\bf 31} (2014) 555--565.

\bibitem{ChaeW}
  \newblock D.~Chae, S.~Weng,
  \newblock {\it Singularity formation for the incompressible Hall-MHD equations without resistivity},
  \newblock Ann. Inst. H. Poincar\'e Anal. Non Lin\'eaire {\bf 33} (2016) 1009--1022.

\bibitem{Cor89}
  \newblock C.~Corduneanu,
  \newblock {\it Almost periodic functions},
  \newblock Chelsea Publishing Compagny, 1989.
  
\bibitem{DML16}
   \newblock E.-C.~D'Avignon, P.-J.~Morrison, M.~Lingam,
   \newblock {\it Derivation of the Hall and extended magnetohydrodynamics brackets},
   \newblock Phys. Plasmas {\bf 23} (2016) 062101.

\bibitem{Dan02}
   \newblock R.~Danchin,
   \newblock {\it Zero Mach number limit for compressible flows with periodic boundary conditions},
   \newblock Amer. J. Math. {\bf 124} (2002) 1153--1219.

\bibitem{DG99}
   \newblock B.~Desjardins, E.~Grenier,
   \newblock {\it Low Mach number limit of viscous compressible flows in the whole space},
   \newblock Proc. R. Soc. Lond.  A {\bf 455} (1999) 2271--2279.
 
\bibitem{DLZ}
  \newblock R.~Duan, Q.~Liu, C.~Zhu,
  \newblock {\it The Cauchy problem on the compressible two-fluids Euler--Maxwell equations},
  \newblock SIAM J. Math. Anal. {\bf 44} (2012) 102--133.

\bibitem{DLT}
  \newblock E.~Dumas, Y.~Lu, B.~Texier,
  \newblock {\it Space-time resonances and high-frequency Raman instabilities in the two-fluid Euler--Maxwell System},
  \newblock to appear in Memoirs of the European Mathematical Society (2025).

\bibitem{Ga98}
  \newblock I.~Gallagher,
  \newblock {\it Asymptotics of the solutions of hyperbolic equations with a skew-symmetric perturbation},
  \newblock J. Differential Equations {\bf 150} (1998) 363--384.
  
\bibitem{Ga}
  \newblock I.~Gallagher,
  \newblock {\it R\'esultats r\'ecents sur la limite incompressible},
  \newblock S\'eminaire Bourbaki {\bf 926} (2003) 29--57.
  
\bibitem{GTAM}
  \newblock D.~Grasso, E.~Tassi, H.-M.~Abdelhamid, P.J.~Morrison,
  \newblock {\it Structure and computation of two-dimensional incompressible extended MHD},
  \newblock Phys. Plasmas {\bf 24} (2017) 012110.

\bibitem{Gre97}
  \newblock E.~Grenier,
  \newblock {\it Oscillatory perturbations of the Navier-Stokes equations},
  \newblock J. Math. Pures Appl. {\bf 76} (1997) 477--498.
  
\bibitem{GIP}
  \newblock Y.~Guo, A.~Ionescu, B.~Pausader,
  \newblock {\it Global solutions of the Euler--Maxwell two-fluid system in 3D},
  \newblock Ann. of Math. {\bf 183} (2016) 377--498.

\bibitem{GP04}
  \newblock H.-P.~Goedbloed, S.~Poedts,
  \newblock {\it Principles of magnetohydrodynamics},
  \newblock Cambridge Press, 2004.

\bibitem{ISYG}
  \newblock  S.~Ibrahim, S.~Shen, T.~Yoneda, Y.~Giga,
  \newblock {\it Global well posedness for a two-fluid model},
  \newblock Differential Integral Equations {\bf 31} (2018) 187--214.

\bibitem{JO}
   \newblock I.-J.~Jeong, S.-J.~Oh,
   \newblock {\it On the Cauchy problem for Hall and electron magnetohydrodynamic equations without resistivity I: illposedness near degenerate stationary solutions},
   \newblock Ann. PDE {\bf 8} (2022) 15.

\bibitem{JJLX}
  \newblock  S.~Jiang, Q.~Ju, F.~Li, Z.~Xin,
  \newblock {\it Low Mach number limit for the full compressible magnetohydrodynamic equations with general initial data},
  \newblock Adv. Math. {\bf 259} (2014) 384--420.

\bibitem{KT98}
  \newblock M.~Keel, T.~Tao,
  \newblock {\it Endpoint Strichartz estimates},
  \newblock Amer. J. Math. {\bf 120} (1998) 955--980.
  
\bibitem{KLMLW}
  \newblock I.~Keramidas Charidakos, M.~Lingam, P.J.~Morrison, R.L.~White, A.~Wurm,
  \newblock {\it Action principles for extended magnetohydrodynamic models},
  \newblock Phys. Plasmas {\bf 21} (2014) 092118.

\bibitem{KMo}
  \newblock K.~Kimura, P.J.~Morrison,
  \newblock {\it On energy conservation in extended magnetohydrodynamics},
  \newblock Phys. Plasmas {\bf 21} (2014) 082101.
  
\bibitem{KM}
  \newblock S.~Klainerman, A.~Majda,
  \newblock {\it Singular limits of quasilinear hyperbolic systems with large parameters and the incompressible limit of compressible fluids},
  \newblock Comm. Pure Appl. Math. {\bf 34} (1981) 481--524.

\bibitem{KM82}
  \newblock S.~Klainerman, A.~Majda,
  \newblock {\it Compressible and incompressible fluids},
  \newblock Comm. Pure Appl. Math. {\bf 35} (1982) 629--651.  
  
\bibitem{K}
  \newblock H.~O.~Kreiss,
  \newblock {\it Problems with different time scales for partial differential equations},
  \newblock Comm. Pure Appl. Math., {\bf 33} (1980) 39--439.

\bibitem{LPX}
  \newblock Y.~Li, Y.-J.~Peng, S.~Xi,
  \newblock {\it The combined non-relativistic and quasi-neutral limit of two-fluid Euler--Maxwell equations},
  \newblock Z. Angew. Math. Phys. {\bf 66} (2015) 3249--3265.
  
\bibitem{LPW}
  \newblock M.~Li, X.~Pu, S.~Wang,
  \newblock {\it Quasineutral limit for the compressible two-fluid Euler--Maxwell equations for well-prepared initial data},
  \newblock Electronic Research Archive  {\bf 28} (2020) 879--895.

\bibitem{JSX19}
  \newblock Q.~Ju, S.~Schochet, X.~Xu,
  \newblock {\it Singular limits of the equations of compressible ideal magneto-hydrodynamics in a domain with boundaries},
  \newblock Asymptot. Anal.  {\bf 113} (2019) 137--165.

\bibitem{LMMbis}
  \newblock M.~Lingam, G.~Miloshevich, P.-J.~Morrison,
  \newblock {\it Concomitant Hamiltonian and topological structures of extended magnetohydrodynamics},
  \newblock Phys. Lett. A {\bf 380}  (2016) 2400--2406.
  
\bibitem{LM}
  \newblock P.-L.~Lions, N.~Masmoudi,
  \newblock {\it Incompressible limit for a viscous compressible fluid},
  \newblock J. Math. Pures Appl. {\bf 77} (1998) 585--627.

\bibitem{LTHG}
  \newblock H.~Liu, M.~Hesse, K.~Genestreti, R.~Nakamura, J.L.~Burch, P.A.~Cassak, N.~Bessho, J.P.~Eastwood, T.~Phan, M.~Swisdak, S.~Toledo-Redondo, M.~Hoshino, C.~Norgren, H.~Ji, T.K.M.~Nakamura,
  \newblock{\it Ohm's law, the reconnection rate, and energy Conversion in collisionless magnetic reconnection},
  \newblock Space Science Reviews {\bf 221} (2025) 16.

\bibitem{Lu}
  \newblock V.-R.~L\"ust,
  \newblock {\it \"Uber die ausbreitung von Wellen in einem plasma},
  \newblock Fortschr. Phys.  {\bf 7} (1959) 503--558.

\bibitem{Mas01}
  \newblock N.~Masmoudi,
  \newblock {\it  Incompressible, inviscid limit of the compressible Navier--Stokes system},
  \newblock Ann. Inst. H. Poincar\'e. Anal. Non Lin\'eaire {\bf 18} (2001) 199--224.

\bibitem{Met09}
  \newblock G.~M\'etivier,
  \newblock {\it The mathematics of nonlinear optics},
  \newblock Handbook of Differential Equations: Evolutionary Equations, Vol. V, North-Holland, (2009) 169--313.
  
\bibitem{MS00}
  \newblock G.~M\'etivier, S.~Schochet,
  \newblock {\it Limite incompressible des \'equations d'Euler non isentropiques},
  \newblock S\'eminaire \'E.D.P.  2000-2001, \'Ecole Polytechnique, 2001, Expos\'e No X, 15 pages.

\bibitem{MS}
  \newblock G.~M\'etivier, S.~Schochet,
  \newblock {\it The incompressible limit of the non-isentropic Euler equations},
  \newblock Arch. Ration. Mech. Anal. {\bf 158} (2001) 61--90.

\bibitem{MS03}
  \newblock G.~M\'etivier, S.~Schochet,
  \newblock {\it Averaging theorems for conservative systems and the weakly compressible Euler equations},
  \newblock J. Differential Equations {\bf 187} (2003) 106--183.
  
\bibitem{MLM}
  \newblock M.~Lingam, G.~Miloshevich, P.-J.~Morrison,
  \newblock {\it On the structure and statistical theory of turbulence of extended magnetohydrodynamics},
  \newblock New J. Phys. {\bf 19} (2017) 015007.

\bibitem{PL}
  \newblock Y.-J.~Peng, C.~Liu,
  \newblock {\it Global non-relativistic quasi-neutral limit for a two-fluid Euler--Maxwell system},
  \newblock J. Differential Equations {\bf 385} (2024) 362--394.
  
\bibitem{PXZ}
  \newblock  Y.-J.~Peng, S.~Xi, L.~Zhao,
  \newblock {\it Hall-MHD system as a simplified one-fluid ion model derived from two-fluid Euler--Maxwell equations},
  \newblock hal-04889318 (2025).

\bibitem{Rau12}
  \newblock  J.~Rauch,
  \newblock {\it Hyperbolic partial differential equations and geometric optics},
  \newblock American Mathematical Society, 2012.

\bibitem{SV07}  
  \newblock  J.A.~Sanders, F.~Verhulst, J.~Murdock,
  \newblock {\it Averaging methods in nonlinear dynamical systems},
  \newblock Springer, 2007.
  
\bibitem{SCD09}
  \newblock A.A.~Schekochihin, S.C.~Cowley, W.~Dorland, G.W.~Hammett, G.G.~Howes, E.~Quataert, T.~Tatsuno,
  \newblock {Astrophysical gyrokinetics: kinetic and fluid turbulent cascades in magnetized weakly collisional plasmas},
  \newblock The Astrophysical Journal Supplement, (2009). {\bf 182} (2009) 310--377.

\bibitem{Sch94a}
  \newblock S.~Schochet,
  \newblock {\it Resonant nonlinear geometric optics for weak solutions of conservation laws},
  \newblock  J. Differential Equations {\bf 113} (1994) 473--504.
  
\bibitem{Sch94b}
  \newblock S.~Schochet,
  \newblock {\it Fast singular limits of hyperbolic PDEs},
  \newblock  J. Differential Equations {\bf 114} (1994) 476--512.
  
\bibitem{S07}
  \newblock S.~Schochet,
  \newblock {\it The mathematical theory of the incompressible limit in fluid dynamics},
  \newblock Handbook of Mathematical Fluid Dynamics, Vol. IV, North-Holland, Amsterdam, (2007) 123--157.

\bibitem{SMA}
  \newblock  N.E.~Shorba, A.A.~Mahmoud, H.M.~Abdelhamid,
  \newblock {\it Incompressible extended magnetohydrodynamics waves: implications of electron inertia},
  \newblock Phys. Fluids {\bf 36} (2024) 097108.

\bibitem{Uka86}
  \newblock  S.~Ukai,
  \newblock {\it The incompressible limit and the initial layer of the compressible Euler equation},
  \newblock J. Math. Kyoto Univ. {\bf 26} (1986) 323--331.
  
\bibitem{XXK}
  \newblock J.~Xu, J.~Xiong, S.~Kawashima,
  \newblock {\it Global well-posedness in critical Besov spaces for two-fluid Euler--Maxwell Equations},
  \newblock SIAM J. Math. Anal. {\bf 45} (2013) 1422--1447.

}
\end{thebibliography}
\end{document}